\documentclass[reqno,a4paper,12pt]{amsart} 

\usepackage{amsmath,amscd,amsfonts,amssymb}
\usepackage{mathrsfs,dsfont}
\usepackage{color}
\usepackage{mathtools}
\usepackage{hyperref}
\usepackage{tikz-cd}

\usepackage{tikz}
\usetikzlibrary{decorations.pathreplacing}

\numberwithin{equation}{section}
\numberwithin{figure}{section}

\def\R{\mathbb{R}}
\def\C{\mathbb{C}}

\def\Z{\mathbb{Z}}
\def\N{\mathbb{N}}

\def\one{\mathds{1}}
\def\eps{\varepsilon}

\renewcommand\leq{\leqslant}
\renewcommand\geq{\geqslant}

\renewcommand\hat{\widehat}

\newcommand\X{\mathrm{X}}

\newcommand{\ft}[1]{\widehat #1}

\newcommand{\Tile}{\operatorname{Tile}}
\newcommand{\Universe}{{\mathfrak{U}}}

\theoremstyle{plain}
\newtheorem{thm}{Theorem}[section]
\newtheorem{theorem}[thm]{Theorem}

\newtheorem{lemma}[thm]{Lemma}
\newtheorem{corollary}[thm]{Corollary}

\newtheorem{proposition}[thm]{Proposition}
\newtheorem{problem}[thm]{Problem}

\newtheorem{conjecture}[thm]{Conjecture}

\newtheorem*{claim*}{Claim}
\newtheorem*{remarks*}{Remarks}
\newtheorem*{remark*}{Remark}
\newtheorem{remark}[thm]{Remark}
\newtheorem{example}[thm]{Example}

\theoremstyle{definition}
\newtheorem{definition}[thm]{Definition}
\newtheorem*{definition*}{Definition}

\newcommand{\thmref}[1]{Theorem~\ref{#1}}
\newcommand{\secref}[1]{Section~\ref{#1}}
\newcommand{\subsecref}[1]{Subsection~\ref{#1}}
\newcommand{\lemref}[1]{Lemma~\ref{#1}}
\newcommand{\defref}[1]{Definition~\ref{#1}}
\newcommand{\propref}[1]{Proposition~\ref{#1}}
\newcommand{\probref}[1]{Problem~\ref{#1}}

\newcommand{\conjref}[1]{Conjecture~\ref{#1}}

\newcommand{\remref}[1]{Remark~\ref{#1}}
\newcommand{\figref}[1]{Figure~\ref{#1}}

\newenvironment{enumerate-math}
{\begin{enumerate}
		\addtolength{\itemsep}{5pt}
		}
	{\end{enumerate}}

\newenvironment{enumerate-text}
{\begin{enumerate}
		\addtolength{\itemsep}{5pt}
		}
	{\end{enumerate}}

\begin{document}

	\title[Undecidable tilings with only two tiles]
	{Undecidable translational tilings with only two tiles, or one nonabelian tile}
	
	\author{Rachel Greenfeld}
	\address{UCLA Department of Mathematics, Los Angeles, CA 90095-1555.}
	\email{greenfeld@math.ucla.edu}
	\author{Terence Tao}
	\address{UCLA Department of Mathematics, Los Angeles, CA 90095-1555.}
	\email{tao@math.ucla.edu}

	\subjclass[]{}
	\date{}
	
	
	\begin{abstract} We construct an example of a group $G = \Z^2 \times G_0$ for a finite abelian group $G_0$, a subset $E$ of $G_0$, and two finite subsets $F_1,F_2$ of $G$, such that it is undecidable in ZFC whether $\Z^2\times E$ can be tiled by translations of $F_1,F_2$.  In particular, this implies that this tiling problem is \emph{aperiodic}, in the sense that (in the standard universe of ZFC) there exist translational tilings of $E$ by the tiles $F_1,F_2$, but no periodic tilings.  Previously, such aperiodic or undecidable translational tilings were only constructed for sets of eleven or more tiles (mostly in $\Z^2$).  A similar construction also applies for $G = \Z^d$ for sufficiently large $d$.  If one allows the group $G_0$ to be non-abelian, a variant of the construction produces an undecidable translational tiling with only one tile $F$.
		
	The argument proceeds by first observing that a single tiling equation is able to encode an arbitrary system of tiling equations, which in turn can encode an arbitrary system of certain functional equations once one has two or more tiles.  In particular, one can use two tiles to encode tiling problems for an arbitrary number of tiles.
	\end{abstract}

	\maketitle
	
	
\section{Introduction}

\subsection{A note on set-theoretic foundations}

In this paper we will be discussing questions of decidability in the Zermelo--Frankel--Choice (ZFC) axiom system of set theory.  As such, we will sometimes have to make distinctions between the \emph{standard universe}\footnote{Here we of course make the metamathematical assumption that the standard universe exists, so that in particular ZFC is consistent.  By the second G\"odel incompleteness theorem, this latter claim, if true, cannot be proven within ZFC itself.} $\Universe$ of ZFC, in which for instance the natural numbers $\N = \N_\Universe$ are the standard natural numbers $\{0,1,2,\dots\}$, the integers $\Z = \Z_\Universe$ are the standard integers $\{0, \pm 1, \pm 2,\dots \}$, and so forth, and also \emph{nonstandard universes} $\Universe^*$ of ZFC, in which the model $\N_{\Universe^*}$ of the natural numbers may possibly admit some nonstandard elements not contained in the standard natural numbers $\N_\Universe$, and similarly for the model $\Z_{\Universe^*}$ of the integers in this universe.  However, every standard natural number $n = n_\Universe \in \N$ will have a well-defined counterpart $n_{\Universe^*} \in \N_{\Universe^*}$ in such universes, which by abuse of notation we shall usually identify with $n$; similarly for standard integers. 

If $S$ is a first-order sentence in ZFC, we say that $S$ is (logically) \emph{undecidable} (or \emph{independent of ZFC}) if it cannot be proven within the axiom system of ZFC.  By the G\"odel completeness theorem, this is equivalent to $S$ being true in some universes of ZFC while being false in others.  For instance, if $S$ is a undecidable sentence that involves the group $\Z^d$ for some standard natural number $d$, it could be that $S$ holds for the standard model $\Z^d = \Z_\Universe^d$ of this group, but fails for some non-standard model $\Z_{\Universe^*}^d$ of the group.

\begin{remark}\label{algo-logic}  In the literature the closely related concept of \emph{algorithmic undecidability} from computability theory is often used.  By a \emph{problem} $S(x), x \in X$ we mean a sentence $S(x)$ involving a parameter $x$ in some range $X$ that can be encoded as a binary string. Such a problem is \emph{algorithmically undecidable} if there is no Turing machine $T$ which, when given $x \in X$ (encoded as a binary string) as input, computes the truth value of $S(x)$ (in the standard universe) in finite time.  One relation between the two concepts is that if the problem $S(x), x \in X$ is algorithmically undecidable then there must be at least one instance $S(x_0)$ of this problem with $x_0 \in X$ that is logically undecidable, since otherwise one could evaluate the truth value of a sentence $S(x)$ for any $x \in X$ by running an algorithm to search for proofs or disproofs of $S(x)$.  Our main results on logical undecidability can also be modified to give (slightly stronger) algorithmic undecidability results; see Remark \ref{algo} below.  However, we have chosen to use the language of logical undecidability here rather than algorithmic undecidability, as the former concept can be meaningfully applied to individual tiling equations, rather than a tiling problem involving one or more parameters $x$.
\end{remark}

In order to describe various mathematical assertions as first-order sentences in ZFC, it will be necessary to have the various parameters of these assertions presented in a suitably ``explicit'' or ``definable'' fashion.  In this paper, this will be a particular issue with regards to finitely generated abelian groups $G = (G,+)$.  Define an \emph{explicit finitely generated abelian group} to be a group of the form
\begin{equation}\label{explicit-group}
\Z^d \times \Z_{N_1} \times \dots \times \Z_{N_m}
\end{equation}
for some (standard) natural numbers $d, m$ and (standard) positive integers $N_1,\dots,N_m$, where we use $\Z_N \coloneqq \Z/N\Z$ to denote the standard cyclic group of order $N$.  For instance, $\Z^2 \times \Z_{21}^{20}$ is an explicit finitely generated abelian group.  We define the notion of a \emph{explicit finite abelian group} similarly by omitting the $\Z^d$ factor.  From the classification of finitely generated abelian groups, we know that (in the standard universe $\Universe$ of ZFC) every finitely generated abelian group is (abstractly) isomorphic to an explicit finitely generated abelian group, but the advantage of working with explicit finitely generated abelian groups is that such groups $G$ are definable in ZFC, and in particular have counterparts $G_{\Universe^*}$ in all universes $\Universe^*$ of ZFC, not just the standard universe $\Universe$.

\subsection{Tilings by a single tile}

If $G$ is an abelian group and $A,F$ are subsets of $G$, we define the set $A \oplus F$ to be the set of all sums $a+f$ with $a \in A, f \in F$ if all these sums are distinct, and leave $A \oplus F$ undefined if the sums are not distinct.  Note that from our conventions we have $A \oplus F = \emptyset$ whenever one of $A,F$ is empty.  Given two sets $F, E$ in $G$, we let $\Tile(F; E)$ denote the \emph{tiling equation}\footnote{In this paper we use tiling to refer exclusively to \emph{translational} tilings, thus we do not permit rotations or reflections of the tile $F$. Also, we adopt the convention that an equation such as \eqref{tile-eq} is automatically false if one or more of the terms in that equation is undefined.}
\begin{equation}\label{tile-eq}
\X \oplus F = E
\end{equation}
where we view the tile $F$ and the set $E$ to be tiled as given data and the indeterminate variable $\X$ denotes an unknown subset of $G$.  We will be interested in the question of whether this tiling equation $\Tile(F;E)$ admits solutions $\X = A$, and more generally what the space 
$$ \Tile(F;E)_{\Universe} \coloneqq \{ A \subset G: A \oplus F = E \}$$ 
of solutions to $\Tile(F;E)$ looks like.  Later on we will generalize this situation by considering systems of tiling equations rather than just a single tiling equation, and also allow for multiple tiles $F_1,\dots,F_J$ rather than a single tile $F$.

We will focus on tiling equations in which $G$ is a finitely generated abelian group, $F$ is a finite subset of $G$, and $E$ is a subset of $G$ which is \emph{periodic}, by which we mean\footnote{We caution that in some literature, the term ``periodic'' instead refers to sets that are unions of cosets of some non-trivial cyclic subgroup of $G$; in our notation, we would refer to such sets as being one-periodic.  For instance, if $G = \Z^2$ and $A$ was an arbitrary subset of $\Z$, then $A \times \Z$ would be one-periodic, but not necessarily periodic in the sense adopted in this paper.  The notion of an aperiodic tiling  is similarly modified in some of the literature, and the notion of aperiodicity used here (see \defref{undecide}(ii)) is sometimes referred to as ``weak aperiodicity''.} that $E$ is a finite union of cosets of some finite index subgroup of $G$. In order to be able to talk about the decidability of such tiling problems we will need to restrict further by requiring that $G$ is an explicit finitely generated abelian group in the sense \eqref{explicit-group} discussed previously.  The finite set $F$ can then  be described explicitly in terms of a finite number of standard integers; for instance, if $F$ is a finite subset of $\Z^2 \times \Z_N$, then one can write it as
$$ F = \{ (a_1,b_1,c_1 \hbox{ mod } N), \dots, (a_k,b_k,c_k \hbox{ mod } N) \}$$
for some standard natural number $k$ and some standard integers $a_1$ , $\dots$, $a_k$, $b_1$, $\dots$, $b_k$, $c_1$, $\dots$, $c_k$.  Thus $F$ is now a definable set in ZFC and has counterparts $F_{\Universe^*}$ in every universe $\Universe^*$ of ZFC.  Similarly, a periodic subset $E$ of an explicit finitely generated abelian group $\Z^d \times \Z_{N_1} \times \dots \times \Z_{N_m}$ can be written as
$$ E = S \oplus ((r\Z^d) \times \Z_{N_1} \times \dots \times \Z_{N_m})$$
for some standard natural number $r$ and some finite subset $S$ of $G$; thus $E$ is also definable and has counterparts $E_{\Universe^*}$ in every universe $\Universe^*$ of ZFC.  One can now consider the solution space
$$ \Tile(F;E)_{\Universe^*} \coloneqq \{ A \subset G_{\Universe^*}: A \oplus F_{\Universe^*} = E_{\Universe^*} \}$$ 
to $\Tile(F;E)$ in any universe $\Universe^*$ of ZFC.

We now consider the following two properties of the tiling equation $\Tile(F;E)$.

\begin{definition}[Undecidability and aperiodicity]\label{undecide}  Let $G$ be an (explicit) finitely generated abelian group, $F$ a finite subset of $G$, and $E$ a periodic subset of $G$.
\begin{itemize}
\item[(i)]  We say that the tiling equation $\Tile(F;E)$ is \emph{undecidable} if the assertion that there exists a solution $A \subset G$ to $\Tile(F;E)$, when phrased as a first-order sentence in ZFC, is not provable within the axiom system of ZFC.  By the G\"odel completeness theorem, this is equivalent to the assertion that $\Tile(F;E)_{\Universe^*}$ is empty for some\footnote{In fact, in the specific context of undecidable tiling equations, one can show that $\Tile(F;E)_{\Universe}$ is non-empty for the \emph{standard} universe $\Universe$; see Appendix \ref{wang-app}.} universes $\Universe^*$ of ZFC, but non-empty for some other universes.  We say that the tiling equation $\Tile(F;E)$ is \emph{decidable} if it is not undecidable.
\item[(ii)]  We say that the tiling equation $\Tile(F;E)$ is \emph{aperiodic} if, when working within the standard universe $\Universe$ of ZFC, the equation $\Tile(F;E)$ admits a solution $A \subset G$, but that none of these solutions are periodic.  That is to say, $\Tile(F;E)_\Universe$ is non-empty but contains no periodic sets.
\end{itemize}
\end{definition}

\begin{example}\label{simple-tile}  Let $G$ be the explicit finitely generated abelian group $G \coloneqq \Z^2$, let $F \coloneqq \{0,1\}^2$, and let $E \coloneqq \Z^2$.  The tiling equation $\Tile(F;E)$ has multiple solutions in the standard universe $\Universe$ of ZFC; for instance, given any (standard) function $a \colon \Z \to \{0,1\}$, the set
$$ A \coloneqq \{ (n, a(n) + m): n,m \in 2\Z \}$$
solves the tiling equation $\Tile(F;E)$ and is thus an element of $\Tile(F;E)_\Universe$.  Most of these solutions will not be periodic, but for instance if one selects the function $a \equiv 0$ (so that $A = (2\Z)^2$) then one obtains a periodic tiling.  This latter tiling is definable and thus has a counterpart in every universe $\Universe^*$ of ZFC, and we conclude that in this case the tiling equation $\Tile(F;E)$ is decidable and not aperiodic.
\end{example}

\begin{remark} The notion of aperiodicity of a tiling equation $\Tile(F;E)$ is only interesting when $E$ is itself periodic, since if $A \oplus F = E$ and $A$ is periodic then $E$ must necessarily be periodic also.
\end{remark}

A well-known argument of Wang (see \cite{Ber, R}) shows that if a tiling equation $\Tile(F;E)$ is not aperiodic, then it is decidable; contrapositively, if a tiling equation is undecidable, then it must also be aperiodic.  From this we see that any undecidable tiling equation must admit (necessarily non-periodic) solutions in the standard universe of ZFC (because the tiling equation is aperiodic), but (by the completeness theorem) will not admit solutions at all in some other (nonstandard) universes of ZFC.  For the convenience of the reader we review the proof of this assertion (generalized to multiple tiles, and to arbitrary periodic subsets $E$ of explicit finitely generated abelian groups $G$) in Appendix \ref{wang-app}.

\subsection{The periodic tiling conjecture}\label{subsub-ptc} The following conjecture was proposed in the case\footnote{Strictly speaking, Lagarias and Wang posed an analogue of this conjecture for $E=G=\R^d$, see \subsecref{problem_Rd}.}  $E=G=\Z^d$ by Lagarias and Wang \cite{LW} and also previously appears implicitly in \cite[p. 23]{grunbaum-shephard}:

\begin{conjecture}[Periodic tiling conjecture]\label{ptc}  Let $G$ be an explicit finitely generated abelian group, let $F$ be a finite non-empty subset of $G$, and let $E$ be a periodic subset of $G$.  Then $\Tile(F;E)$ is not aperiodic.
\end{conjecture}

By the previous discussion, Conjecture \ref{ptc} implies that the tiling equation $\Tile(F;E)$ is decidable for every $F,E,G$ obeying the hypotheses of the conjecture.

The following progress is known towards the periodic tiling conjecture:

\begin{itemize}
\item Conjecture \ref{ptc} is trivial when $G$ is a finite abelian group, since in this case all subsets of $G$ are periodic.
\item When $E=G=\Z$, Conjecture \ref{ptc} was established by Newman \cite{N} as a consequence of the pigeonhole principle. In fact, the argument shows that \emph{every} set in $\Tile(F;\Z)_\Universe$ is periodic. As we shall review in Section \ref{one-tile} below, the argument also extends to the case $G = \Z \times G_0$ for an (explicit) finite abelian group $G_0$, and to an arbitrary periodic subset $E$ of $G$. See also the results in Section \ref{single-multi} for some additional properties of one-dimensional tilings.
\item When $E=G=\Z^2$, Conjecture \ref{ptc}  was established by Bhattacharya \cite{BH} using ergodic theory methods (viewing $\Tile(F; \Z^2)_\Universe$ as a dynamical system using the translation action of $\Z^2$).  In our previous paper \cite{GT} we gave an alternative proof of this result, and generalized it to the case where $E$ is a periodic subset of $G = \Z^2$. In fact,  we strengthen the previous result of Bhattacharya, by showing  that every set in $\Tile(F,E)_\Universe$ is \textit{weakly periodic} (a disjoint union of finitely many one-periodic sets).  In the case of polyominoes (where $F$ is viewed as a union of unit squares whose boundary is a simple closed curve), the conjecture was previously established in \cite{bn}, \cite{gbn}\footnote{In fact, they showed that when $F$ is a polyomino, every set in $\Tile(F;\Z^2)_\Universe$ is one-periodic.} and decidability was established even earlier in \cite{wvl}.
\end{itemize}

The conjecture remains open in other cases; for instance, the case $E=G=\Z^3$ or the case $E=G=\Z^2 \times \Z_N$ for an arbitrary natural number $N$, are currently unresolved, although we hope to report on some results in these cases in forthcoming work. In \cite{szegedy} it was shown that Conjecture \ref{ptc} for $E=G=\Z^d$ was true whenever the cardinality $|F|$ of $F$ was prime, or less than or equal to four.

\subsection{Tilings by multiple tiles}

It is natural to ask if Conjecture \ref{ptc} extends to tilings by multiple tiles.  Given subsets $F_1,\dots,F_J,E$ of a group $G$, we use $\Tile(F_1,\dots,F_J;E) = \Tile((F_j)_{j=1}^J; E)$ to denote the tiling equation\footnote{See Section \ref{notation-sec} for our conventions on precedence of operations such as $\oplus$ and $\uplus$.  In the language of convolutions, one can also write this tiling equation as $\one_{\X_1} * \one_{F_1} + \dots + \one_{\X_J} * \one_{F_J} = \one_E$, where we use $\one_A$ to denote the indicator function of $A$.}
\begin{equation}\label{multitile}
\biguplus_{j=1}^J \X_j \oplus F_j = E,
\end{equation}
where $A \uplus B$ denotes the disjoint union of $A$ and $B$ (equal to $A \cup B$ when $A,B$ are disjoint, and undefined otherwise).  As before we view $F_1,\dots,F_J,E$ as given data for this equation, and $\X_1,\dots,\X_J$ are indeterminate variables representing unknown tiling sets in $G$.  If $G$ is an explicit finitely generated group, $F_1,\dots,F_J$ are finite subsets of $G$, and $E$ is a periodic subset of $G$, we can define the solution set
$$
\Tile(F_1,\dots,F_J;E)_\Universe \coloneqq \left\{ (A_1,\dots,A_J) \colon A_1,\dots,A_J \subset G; \biguplus_{j=1}^J A_j \oplus F_j = E \right\}$$
and more generally for any other universe $\Universe^*$ of ZFC we have
$$
\Tile(F_1,\dots,F_J;E)_{\Universe^*} \coloneqq \left\{ (A_1,\dots,A_J) \colon A_1,\dots,A_J \subset G_{\Universe^*}; \biguplus_{j=1}^J A_j \oplus (F_j)_{\Universe^*} = E_{\Universe^*} \right\}.$$

We extend Definition \ref{undecide} to multiple tilings in the natural fashion:

\begin{definition}[Undecidability and aperiodicity for multiple tiles]\label{undecide-mult}  Let $G$ be an explicit finitely generated abelian group, $F_1,\dots,F_J$ be finite subsets of $G$ for some standard natural number $J$, and $E$ a periodic subset of $G$.
\begin{itemize}
\item[(i)]  We say that the tiling equation $\Tile(F_1,\dots,F_J;E)$ is \emph{undecidable} if the assertion that there exist subsets $A_1,\dots,A_J \subset G$ solving $\Tile(F_1,\dots,F_J;E)$, when phrased as a first-order sentence in ZFC, is not provable within the axiom system of ZFC. By the G\"odel completeness theorem, this is equivalent to the assertion that $\Tile(F_1,\dots,F_J;E)_{\Universe^*}$ is non-empty for some universes $\Universe^*$ of ZFC, but empty for some other universes. We say that $\Tile(F_1,\dots,F_J;E)$ is \emph{decidable} if it is not undecidable.
\item[(ii)]  We say that the tiling equation $\Tile(F_1,\dots,F_J;E)$ is \emph{aperiodic} if, when working within the standard universe $\Universe$ of ZFC, the equation $\Tile(F_1,\dots,F_J;E)$ admits a solution $A_1,\dots,A_J \subset G$, but there are no solutions for which all of the $A_1,\dots,A_J$ are periodic. That is to say, $\Tile(F_1,\dots,F_J;E)_\Universe$ is non-empty but contains no tuples of periodic sets.
\end{itemize}
\end{definition}

As in the single tile case, undecidability implies aperiodicity; see Appendix \ref{wang-app}. The argument of Newman that resolves the one-dimensional case of Conjecture \ref{ptc} also shows that for (explicit) one-dimensional groups $G = \Z \times G_0$, every tiling equation $\Tile(F_1,\dots,F_J;E)$ is not aperiodic (and thus also decidable); see Section \ref{one-tile}.

However, in marked contrast to what Conjecture \ref{ptc} predicts to be the case for single tiles, it is known that a tiling equation $\Tile(F_1,\dots,F_J;E)$ \emph{can} be aperiodic or even undecidable when $J$ is large enough.  In the model case $E=G=\Z^2$, an aperiodic tiling equation $\Tile(F_1,\dots,F_J;\Z^2)$ was famously constructed\footnote{Strictly speaking, Berger's construction was for the closely related \emph{domino problem} (or \emph{Wang tiling problem}), but it was shown by Golomb \cite{golomb} shortly afterwards that this construction also implies undecidability for the translational tiling problem.  Similar considerations apply to several of the other constructions listed in Table \ref{J-table}.} by Berger \cite{Ber} with $J = 20426$, and an undecidable tiling was also constructed by a modification of the method with an unspecified value of $J$.  A simplified proof of this latter fact was given by Robinson \cite{R}, who also constructed a collection of $J=36$ tiles was constructed in which a related \emph{completion problem} was shown to be undecidable.  The value of $J$ for either undecidable examples or aperiodic examples has been steadily lowered over time; see Table \ref{J-table} for a partial list.  We refer the reader to the recent survey \cite{jv} for more details of these results.  To our knowledge, the smallest known value of $J$ for an aperiodic tiling equation $\Tile(F_1,\dots,F_J;\Z^2)$ is $J=8$, by Ammann, Gr\"unbaum, and Shephard \cite{ags}. The smallest known value of $J$ for a tiling equation $\Tile(F_1,\dots,F_J;\Z^2)$ that was explicitly constructed and shown to be undecidable is $J=11$, due to Ollinger \cite{ollinger-2}.

\begin{remark}
As Table \ref{J-table} demonstrates, many of these constructions were based on a variant of a tile set in $\Z^2$ known as a set of \emph{Wang tiles}, but in \cite{jr} it was shown that Wang tile constructions cannot create aperiodic (or undecidable) tile sets for any $J < 11$.  
\end{remark}

Analogous constructions in higher dimensions were obtained for $E=G=\Z^3$ (or more precisely $\R^3$) in \cite{danzer}, \cite{schmitt}, \cite{CK} and for $E=G=\Z^n$ (or more precisely $\R^n$), $n \geq 3$ in \cite{gs-higher}. 
\begin{table}[ht]
\centering
\begin{tabular}[t]{lll}
$J$ & Author & Type \\
\hline
$20426$ & Berger \cite{Ber} & aperiodic [undecidable] (W) \\
$104$ & Robinson \cite{ams} & aperiodic (W) \\
$104$ & Ollinger \cite{ollinger} & aperiodic [undecidable] (W) \\
$103$ & Berger \cite{Ber-thesis} & aperiodic (W) \\
$86$ & Knuth \cite{knuth} & aperiodic (W) \\
$56$ & Robinson \cite{R} & aperiodic (W) \\
$52$ & Robinson \cite{poizat} & aperiodic (W) \\
$40$ & Lauchli \cite{wang} & aperiodic (W) \\
$36$ & Robinson \cite{R} & completion-undecidable (W) \\
$32$ & Robinson \cite{grunbaum-shephard} & aperiodic (W) \\
$24$ & Gr\"unbaum--Shephard \cite{grunbaum-shephard} & aperiodic (W) \\
$24$ & Robinson \cite{grunbaum-shephard} & aperiodic (W) \\
$16$ & Ammann et al. \cite{ags} & aperiodic (W) \\
$14$ & Kari \cite{kari} & aperiodic (W) \\
$13$ & Culik \cite{culik} & aperiodic (W) \\
$12$ & Socolar--Taylor \cite{socolar-taylor} & aperiodic \\
$11$ & Jeandel--Rao \cite{jr} & aperiodic (W)\\
$11$ & Ollinger \cite{ollinger-2} & undecidable \\
$8$ & Ammann et al. \cite{ags} & aperiodic \\
$8$ & Goodman-Strauss \cite{goodman} & aperiodic \\
\hline
{\bf $2^*$} & {\bf Theorems \ref{main}. \ref{main'}} & {\bf undecidable}\\
{\bf $1^{**}$} & {\bf Theorem \ref{onetile}} & {\bf undecidable} \\
\hline
\end{tabular}
\caption{Selected constructions of aperiodic or undecidable tiling equations. This list is primarily adapted from \cite{jr}, and incorporates from that reference some corrections to the values of $J$ in several lines of this table.  Constructions labeled (W) arise from a Wang tile construction.  The constructions marked ``aperiodic [undecidable]'' give aperiodic tilings for the specified value of $J$, and an undecidable tiling for an unspecified value of $J$.  The asterisk for our results in Theorems \ref{main}, \ref{main'} denotes the fact that we are replacing $\Z^2$ by $\Z^2 \times E_0$ for some subset $E_0$ of an explicit finite abelian group $G_0$, or by a periodic subset of some high-dimensional lattice $\Z^d$.  The double asterisk indicates that the tiling is nonabelian.  For some other notable constructions of aperiodic or undecidable tiling equations (but with values of $J$ that are either not explicitly stated, or larger than other contemporary constructions), see \cite{jr}, \cite{goodman}, \cite{jv}.}
\label{J-table}
\end{table}%

\subsection{Main results}

Our first main result is that one can in fact obtain undecidable (and hence aperiodic) tiling equations for $J$ as small as $2$, at the cost of enlarging $E$ from $\Z^2$ to $\Z^2 \times E_0$ for some subset $E_0$ of a (explicit) finite abelian group $G_0$.

\begin{theorem}[Undecidable tiling equation with two tiles in $\Z^2\times G_0$]\label{main}  There exists an explicit finite abelian group $G_0$, a subset $E_0$ of $G_0$, and finite non-empty subsets $F_1,F_2$ of $\Z^2 \times G_0$ such that the tiling equation $\Tile(F_1,F_2; \Z^2 \times E_0)$ is undecidable (and hence aperiodic).
\end{theorem}

The proof of \thmref{main} goes on throughout Sections 3--8. In \secref{extension}, by ``pulling back'' the proof of \thmref{main}, we prove the following analogue  in $\Z^d$.

\begin{theorem}[Undecidable tiling equation with two tiles in $\Z^d$]\label{main'}  There exists an explicit $d>1$, a periodic subset $E$ of $\Z^d$, and finite non-empty subsets $F_1,F_2$ of $\Z^d$ such that the tiling equation $\Tile(F_1,F_2; E)$ is undecidable (and hence aperiodic).
\end{theorem}

\begin{remark}\label{extend_Rd}
One can further extend our construction in Theorem \ref{main'} to the Euclidean space $\R^d$, as follows. First,  replace each tile $F_j\subset \Z^d$, $j=1,2$, with a finite union $\tilde F_j$ of unit cubes centered in $F_j$, and similarly replace $E \subset \Z^d$ with a periodic set $\tilde E \subset \R^d$. Next, in order to make the construction rigid in the Euclidean space,  add ``bumps'' on the sides (as in the proof of \lemref{rigid-tile}).  When one does so, the only tilings of $\tilde E$ by the $\tilde F_1, \tilde F_2$ arise from tilings of $E$ by $F_1,F_2$, possibly after applying a translation, and hence the undecidability of the former tiling problem is equivalent to that of the latter.  
\end{remark}

Our construction can in principle give a completely explicit description of the sets $G_0, E_0, F_1, F_2$, but they are quite complicated (and the group $G_0$ is  large), and we have not attempted to optimize the size and complexity of these sets in order to keep the argument as conceptual as possible.

\begin{remark}
Our argument establishes an encoding for \emph{any} tiling problem $\Tile(F_1,\dots,F_J; \Z^2)$ with arbitrary number of tiles in $\Z^2$ as a tiling problem with two tiles in $\Z^2\times G_0$. However, in order to prove \thmref{main} we only need to be able to encode Wang tilings.
\end{remark}

\begin{remark}\label{algo}  A slight modification of the proof of Theorem \ref{main} also establishes the slightly stronger claim that the decision problem of whether the tiling equation $\Tile(F_1,F_2; \Z^2 \times E_0)$ is solvable for a given finite abelian group $G_0$, given finite non-empty subsets $F_1,F_2 \subset \Z^2 \times G_0$ and $E_0 \subset G_0$, is \emph{algorithmically} undecidable.  Similarly for Theorems \ref{main'}, \ref{onetile} below.  This is basically because the original undecidability result of Berger \cite{Ber} that we rely on is also phrased in the language of algorithmic undecidability; see Footnote \ref{berfoot} in Section \ref{tiling}.  We leave the details of the appropriate modification of the arguments in the context of algorithmic decidability to the interested reader.  
\end{remark}

Theorem \ref{main} supports the  belief\footnote{As a dissenting view, it was conjectured in \cite{gbn} that the translational tiling problem for $\Z^2$ with $J=2$ is (algorithmically) decidable and not aperiodic.} that the tiling problem is considerably less well behaved  for $J \geq 2$ than it is for $J=1$.  As another instance of this belief, the $J=1$ tilings enjoy a dilation symmetry (see \cite[Proposition 3.1]{BH}, \cite[Lemma 3.1]{GT}, \cite{tij}) that have no known analogue for $J \geq 2$.  We present a further distinction between the $J=1$ and $J \geq 2$ situations in Section \ref{single-multi} below, where we show that in one dimension the $J=1$ tilings exhibit a certain partial rigidity property that is not present in the $J \geq 2$ setting, and makes  any attempt to extend our methods of proof of Theorem \ref{main} to the $J=1$ case  difficult.  On the other hand, if one allows the group $G_0$ to be \emph{nonabelian}, then we can reduce the two tiles in Theorem \ref{main} to a single tile: see Section \ref{nonab-sec}.

\subsection{Overview of proof}

We now discuss the proof of Theorem \ref{main}; the proofs of Theorems \ref{main'}, \ref{onetile} are proven by modifications of the method and are discussed in Sections \ref{extension}, \ref{nonab-sec} respectively.

The arguments proceed by a series of reductions in which we successively replace the tiling equation \eqref{tile-eq} by a more tractable system of equations; see Figure \ref{fig:logic}.  

\begin{figure}
    \centering
    \begin{tikzcd}
& \hbox{\cite{Ber}} \arrow[d, Rightarrow, "\hbox{\S \ref{tiling}}"] \\
& \textrm{Thm. \ref{main-boolean}} \arrow[d, Rightarrow, "\hbox{\S \ref{boolean-sec}}"] \\
& \textrm{Thm. \ref{antipode}} \arrow[d, Rightarrow, "\hbox{\S \ref{boolean-sec}}"] \\
& \textrm{Thm. \ref{main-linear}} \arrow[dl, Rightarrow, "\hbox{\ref{linear-encoding}}"'] \arrow[d, Rightarrow, "\hbox{\S \ref{linear-sec}}"] \\
\textbf{Thm. \ref{onetile}}  & \textrm{Thm. \ref{main-hamming}} \arrow[dl, Rightarrow, "\hbox{\S \ref{hamming-sec}}"] \arrow[dr, Rightarrow, "\hbox{\S \ref{extension}}"'] \\
\textrm{Thm. \ref{main-function}} \arrow[d, Rightarrow, "\hbox{\ref{tile-function}}"] & & \textrm{Thm. \ref{main-function'}} \arrow[d,Rightarrow, "\hbox{\ref{tile-function'}}"] \\
\textrm{Thm. \ref{main-red}}  \arrow[d, Rightarrow, "\hbox{\ref{combine}}"]  && \textrm{Thm. \ref{main-red'}}  \arrow[d, Rightarrow, "\hbox{\ref{combine'}}"] \\
\textbf{Thm. \ref{main}} && \textbf{Thm. \ref{main'}}
\end{tikzcd}
    \caption{The logical dependencies between the undecidability results in this paper (and in \cite{Ber}).  For each implication, there is listed either the section where the implication is proven, or the number of the key proposition or lemma that facilitates the implication. We also remark that Proposition \ref{tile-function'} is proven using Lemma \ref{rigid-tile}, while Proposition \ref{linear-encoding} is proven using Corollary \ref{tiling-system-2}, which in turn follows from Lemma \ref{tiling-system}.}
    \label{fig:logic}
\end{figure}
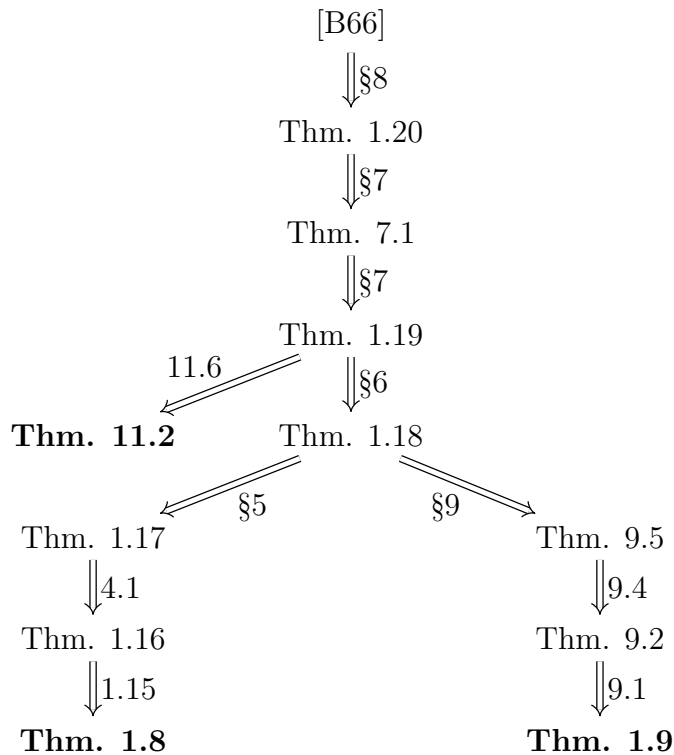

We first extend Definition \ref{undecide-mult} to \emph{systems} of tiling equations.

\begin{definition}[Undecidability and aperiodicity for systems of tiling equations with multiple tiles]\label{undecide-system}  Let $G$ be an explicit finitely generated abelian group, $J, M \geq 1$ be standard natural numbers, and for each $m=1,\dots,M$, let $F^{(m)}_1,\dots,F^{(m)}_J$ be finite subsets of $G$, and let $E^{(m)}$ be a periodic subset of $G$.
\begin{itemize}
\item[(i)]  We say that the system $\Tile(F^{(m)}_1,\dots,F^{(m)}_J;E^{(m)}), m=1,\dots,M$ is \emph{undecidable} if the assertion that there exist subsets $A_1,\dots,A_J \subset G$ that simultaneously solve  $\Tile(F^{(m)}_1,\dots,F^{(m)}_J;E^{(m)})$ for all $m=1,\dots,M$, when phrased as a first-order sentence in ZFC, is not provable within the axiom system of ZFC.  That is to say, the solution set
$$ \bigcap_{m=1}^M \Tile(F^{(m)}_1,\dots,F^{(m)}_J;E^{(m)})_{\Universe^*}$$
is non-empty in some universes $\Universe^*$ of ZFC, and empty in others.  We say that the system is \emph{decidable} if it is not undecidable.
\item[(ii)]  We say that the system
$\Tile(F^{(m)}_1,\dots,F^{(m)}_J;E^{(m)}), m=1,\dots,M$ is \emph{aperiodic} if, when working within the standard universe $\Universe$ of ZFC, this system admits a solution $A_1,\dots,A_J \subset G$, but there are no solutions for which all of the $A_1,\dots,A_J$ are periodic. That is to say, the solution set
$$ \bigcap_{m=1}^M \Tile(F^{(m)}_1,\dots,F^{(m)}_J;E^{(m)})_{\Universe}$$
is non-empty but contains no tuples of periodic sets.
\end{itemize}
\end{definition}

\begin{example}\label{tilex}  Let $G$ be an explicit finitely generated abelian group, and let $G_0$ be a explicit finite abelian group.  The solutions $A$ to the tiling equation $\Tile(\{0\} \times G_0; G \times G_0)$ are precisely those sets which are graphs
\begin{equation}\label{af}
 A = \{ (n, f(n)) : n \in G \}
 \end{equation}
for an arbitrary function $f \colon G \to G_0$.  It is possible to impose additional conditions on $f$ by adding more tiling equations to this ``base'' tiling equation $\Tile(\{0\} \times G_0; G \times G_0)$.  For instance, if in addition $H$ is a subgroup of $G_0$ and $y+H$ is a coset of $H$ in $G_0$, solutions $A$ to the system of tiling equations
$$ \Tile(\{0\} \times G_0; G \times G_0), \Tile(\{0\} \times H; G \times (y+H))$$
are precisely sets $A$ of the form \eqref{af} where the function $f$ obeys the additional\footnote{In this particular case, the former tiling equation is redundant, being a consequence of the latter. However, we choose to retain the former equation in this example to illustrate the principle of imposing additional constraints on the function $f$ by the insertion of additional tiling equations.} constraint $f(n) \in y+H$ for all $n \in G$.    As a further example, if $-y_0, y_0$ are distinct elements of $G_0$, and $h$ is a non-zero element of $G$, then solutions $A$ to the system of tiling equations
$$ \Tile(\{0\} \times G_0; G \times G_0), \Tile(\{0,-h\} \times \{0\}; G \times \{-y_0,y_0\})$$
are precisely sets $A$ of the form \eqref{af} where the function $f$ takes values in $\{-y_0,y_0\}$ and obeys the additional constraint $f(n+h) = -f(n)$ for all $n \in G$. In all three cases one can verify that the system of tiling equations is decidable and not aperiodic.
\end{example}

We then have

\begin{theorem}[Combining multiple tiling equations into a single equation]\label{combine}  Let $J,M \geq 1$, let $G = \Z^d \times G_0$ be an explicit finitely generated abelian group for some explicit finite abelian group $G_0$.  Let $\Z_N$ be a cyclic group with $N>M$, and for each $m=1,\dots,M$ let $F_1^{(m)},\dots,F_J^{(m)}$ be finite non-empty subsets of $G$ and $E_0^{(m)}$ a subset of $G_0$.  Define the finite sets $\tilde F_1,\dots,\tilde F_J \subset G \times \Z_N$ and the set $\tilde E_0 \subset G_0 \times \Z_N$ by 
\begin{equation}\label{fj-stack}
 \tilde F_j \coloneqq \biguplus_{m=1}^M F_j^{(m)} \times \{m\}
 \end{equation}
and
\begin{equation}\label{e0-stack}
 \tilde E_0 \coloneqq \biguplus_{m=1}^M E_0^{(m)} \times \{m\}.
 \end{equation}
\begin{itemize}
\item[(i)]  The system $\Tile(F_1^{(m)},\dots,F_J^{(m)}; \Z^d \times E_0^{(m)}), m=1,\dots,M$ of tiling equations is aperiodic if and only if the tiling equation
$\Tile(\tilde F_1,\dots,\tilde F_J; \Z^d \times \tilde E_0)$ is aperiodic.
\item[(ii)]  The system $\Tile(F_1^{(m)},\dots,F_J^{(m)}; \Z^d \times E_0^{(m)}), m=1,\dots,M$ of tiling equations is undecidable if and only if the tiling equation
$\Tile(\tilde F_1,\dots,\tilde F_J; \Z^d \times \tilde E_0)$ is undecidable.
\end{itemize}
\end{theorem}

Theorem \ref{combine} can be established by easy elementary considerations; see Section \ref{combine-sec}. In view of this theorem, Theorem \ref{main} now reduces to the following statement.

\begin{theorem}[An undecidable system of tiling equations with two tiles in $\Z^2\times G_0$]\label{main-red}  There exists an explicit finite abelian group $G_0$, a standard natural number $M$, and for each $m=1,\dots,M$ there exist finite non-empty sets $F_1^{(m)}, F_2^{(m)} \subset \Z^2 \times G_0$ and $E_0^{(m)} \subset G_0$ such that the system of tiling equations $\Tile(F^{(m)}_1,F^{(m)}_2; \Z^2 \times E_0^{(m)}), m=1,\dots,M$ is undecidable.
\end{theorem}

The ability to now impose an arbitrary number of tiling equations grants us a substantial amount of flexibility.  In Section \ref{function-sec} we will take advantage of this flexibility to replace the system of tiling equations with a system of \emph{functional} equations, basically by generalizing the constructions provided in Example \ref{tilex}.  Specifically, we will reduce Theorem \ref{main-red} to the following statement. 

\begin{theorem}[An undecidable system of functional equations]\label{main-function}  There exists an explicit finite abelian group $G_0$, a standard integer $M \geq 1$, and for each $m=1,\dots,M$ there exist (possibly empty) finite subsets $H_1^{(m)}, H_2^{(m)}$ of $\Z^2 \times \Z_2$ and (possibly empty sets) $F_1^{(m)}, F_2^{(m)}, E^{(m)} \subset G_0$ for $m=1,\dots,M$ such that the question of whether there exist  functions $f_1,f_2 \colon \Z^2 \times \Z_2 \to G_0$ that solve
the system of functional equations
\begin{equation}\label{func-eq}
\biguplus_{h_1 \in H_1^{(m)}} (F_1^{(m)} + f_1(n+h_1)) \uplus
\biguplus_{h_2 \in H_2^{(m)}} (F_2^{(m)} + f_2(n+h_2)) = E^{(m)}
\end{equation}
for all $n \in \Z^2\times \Z_2$ and $m=1,\dots,M$ is undecidable (when expressed as a first-order sentence in ZFC).
\end{theorem}

In the above theorem, the functions $f_1,f_2$ can range freely in the finite group $G_0$.  By taking advantage of the $\Z_2$ factor in the domain, we can restrict $f_1,f_2$ to range instead in a Hamming cube $\{-1,1\}^D \subset \Z_N^D$, which will be more convenient for us to work with, at the cost of introducing an additional sign in the functional equation \eqref{func-eq}.  More precisely, in Section \ref{hamming-sec} we reduce Theorem \ref{main-function} to

\begin{theorem}[An undecidable system of functional equations in the Hamming cube]\label{main-hamming}  There exist  standard integers $N > 2$ and $D,M \geq 1$, and for each $m=1,\dots,M$ there exist shifts $h_1^{(m)}, h_2^{(m)} \in \Z^2$ and (possibly empty sets) $F_1^{(m)}, F_2^{(m)}, E^{(m)} \subset \Z_N^D$ for $m=1,\dots,M$ such that the question of whether there exist functions $f_1,f_2 \colon \Z^2 \to \{-1,1\}^D$ that solve
the system of functional equations
\begin{equation}\label{func-eq-1}
(F_1^{(m)} + \epsilon f_1(n+h_1^{(m)})) \uplus
(F_2^{(m)} + \epsilon f_2(n+h_2^{(m)})) = E^{(m)}
\end{equation}
for all $n \in \Z^2$, $m=1,\dots,M$, and $\epsilon=\pm 1$ is undecidable (when expressed as a first-order sentence in ZFC).
\end{theorem}

The next step is to replace the functional equations \eqref{func-eq-1} with \emph{linear} equations on Boolean functions $f_{j,d} \colon \Z^2 \to \{-1,1\}$ (where we now view $\{-1,1\}$ as a subset of the integers).  More precisely, in Section \ref{linear-sec} we reduce Theorem \ref{main-hamming} to

\begin{theorem}[An undecidable system of linear equations for Boolean functions]\label{main-linear}  There exist  standard integers $D \geq D_0 \geq 1$ and $M_1, M_2 \geq 1$, integer coefficients $a_{j,d}^{(m)} \in \Z$ for $j=1,2$, $d=1,\dots,D$, $m=1,\dots,M_j$, and shifts $h_d \in \Z^2$ for $d=1,\dots,D_0$ such that the question of whether there exist functions $f_{j,d} \colon \Z^2 \to \{-1,1\} \subset \Z$ for $j=1,2$ and $d=1,\dots,D$ that solve
the system of linear functional equations
\begin{equation}\label{func-eq-2}
\sum_{d=1}^D a_{j,d}^{(m)} f_{j,d}(n) = 0
\end{equation}
for all $n \in \Z^2$, $j=1,2$, and $m = 1,\dots,M_j$, as well as the system of linear functional equations
\begin{equation}\label{func-eq-3}
f_{2,d}(n+h_d) = - f_{1,d}(n)
\end{equation}
for all $n \in \Z^2$ and $d=1,\dots,D_0$, is undecidable (when expressed as a first-order sentence in ZFC).
\end{theorem}

Now that we are working with linear equations for Boolean functions, we can encode a powerful class of constraints, namely all local Boolean constraints.  In Section \ref{boolean-sec} we will reduce Theorem \ref{main-linear} to

\begin{theorem}[An undecidable local Boolean constraint]\label{main-boolean}  There exist standard integers $D, L \geq 1$, shifts $h_1,\dots,h_L \in \Z^2$, and a set $\Omega \subset \{-1,1\}^{DL}$ such that the question of whether there exist functions $f_d \colon \Z^2 \to \{-1,1\}$, $d=1,\dots,D$ that solve the constraint
\begin{equation}\label{boolean}
(f_d(n+h_l))_{d=1,\dots,D; l=1,\dots,L} \in \Omega 
\end{equation}
for all $n \in \Z^2$ is undecidable (when expressed as a first-order sentence in ZFC).
\end{theorem}

Finally, in Section \ref{tiling} we use the previously established existence of undecidable translational tile sets to prove Theorem \ref{main-boolean}, and thus Theorem \ref{main}.

\subsection{Acknowledgments}

RG was partially supported by the Eric and Wendy
Schmidt Postdoctoral Award. TT was partially supported by NSF grant DMS-1764034 and by a Simons Investigator Award. We gratefully acknowledge the hospitality and support of the Hausdorff Institute for Mathematics where a significant portion of this research was conducted. 

We thank David Roberts for drawing our attention to  the reference \cite{socolar-taylor}, Hunter Spink for drawing our attention to the reference \cite{glt}, Jarkko Kari for drawing our attention to the references \cite{k92,kp,Lukk}, and Zachary Hunter and Matthew Foreman for further corrections.  We are also grateful to the anonymous referee for several suggestions that improved the exposition of this paper.

\subsection{Notation}\label{notation-sec}

Given a subset $A \subset G$ of an abelian group $G$ and a shift $h \in G$, we define $A+h = h+A\coloneqq \{ a+h: a \in A \}$, $A-h \coloneqq \{ a-h: a \in A \}$, and $-A \coloneqq \{-a: a \in A \}$.

The unary operator $-$ is understood to take precedence over the binary operator $\times$, which in turn takes precedence over the binary operator $\oplus$, which takes precedence over the binary operator $\uplus$.  Thus for instance
$$ A \times -B \oplus -C \times D \uplus E = ((A \times (-B)) \oplus ((-C) \times D)) \uplus E.$$

By slight abuse of notation, any set of integers will be identified with the corresponding set of residue classes in a cyclic group $\Z_N$, if these classes are distinct.  For instance, if $M \leq N$, we identify $\{1,\dots,M\}$ with the residue classes $\{1 \mod N, \dots, M \mod N \} \subset \Z_N$, and if $N>2$, we identify $\{-1,1\}$ with $\{-1 \mod N, 1 \mod N \} \subset \Z_N$.

\section{Periodic tiling conjecture in one dimension}\label{one-tile}

In this section we adapt the pigeonholing argument of Newman \cite{N} to establish

\begin{theorem}[One-dimensional case of periodic tiling conjecture]\label{oned}  Let $G = \Z \times G_0$ for a some explicit finite abelian group $G_0$, let $J \geq 1$ be a standard integer, let $F_1,\dots,F_J$ be finite subsets of $G$, and let $E$ be a periodic subset of $G$.  Then the tiling equation $\Tile(F_1,\dots,F_J; E)$ is not aperiodic (and hence also decidable).
\end{theorem}

We remark that the same argument also applies to systems of tiling equations in one-dimensional groups $\Z \times G_0$; this also follows from the above theorem and Theorem \ref{combine}.

\begin{proof}  Suppose one has a solution $(A_1,\dots,A_J) \in \Tile(F_1,\dots,F_J; E)_\Universe$ to the tiling equation $\Tile(F_1,\dots,F_J; E)$.
To establish the theorem it will suffice to construct a periodic solution $A'_1,\dots,A'_J  \in \Tile(F_1,\dots,F_J; E)_\Universe$ to the same equation.

We abbreviate the ``thickened interval'' $\{ n \in \Z: a \leq n \leq b\} \times G_0$ as $[[a,b]]$ for any integers $a \leq b$.
Since the $F_1,\dots,F_J$ are finite, there exists a natural number $L$ such that $F_1,\dots,F_J \subset [[-L,L]]$.  Since $E$ is periodic, there exists a natural number $r$ such that $E + (n,0) = E$ for all $n \in r\Z$, where we view $(n,0)$ as an element of $\Z \times G_0$.  We can assign each $n \in r\Z$ a ``color'', defined as the tuple
$$ \left( (A_j - (n,0)) \cap [[-L,L]] \right)_{j=1}^J.$$
This is a tuple of $J$ subsets of the finite set $[[-L,L]]$, and thus there are only finitely many possible colors.  By the pigeonhole principle, one can thus find a pair of integers $n_0, n_0+D \in r\Z$ with $D>L$ that have the same color, thus
$$ (A_j - (n_0+D,0)) \cap [[-L,L]] = (A_j - (n_0,0) ) \cap [[-L,L]]$$
or equivalently
\begin{equation}\label{aj-agree}
 A_j \cap [[n_0+D-L,n_0+D+L]] = \left(  A_j \cap [[n_0-L,n_0+L]] \right) + (D,0)
 \end{equation}
 for $j=1,\dots,J$.  

We now define the sets $A'_j$ for $j=1,\dots,J$ by taking the portion $A_j \cap [[n_0,n_0+D-1]]$ of $A_j$ and extending periodically by $D\Z \times \{0\}$, thus
$$ A'_j \coloneqq ( A_j \cap [[n_0,n_0+D-1]]) \oplus D\Z \times \{0\}.$$
Clearly we have the agreement
$$ A'_j \cap [[n_0, n_0+D-1]] = A_j \cap [[n_0, n_0 + D-1]]$$
of $A_j, A'_j$ on $[[n_0,n_0+D-1]]$, but from \eqref{aj-agree} we also have
\begin{align*}
A'_j \cap [[n_0-L,n_0-1]] &= (A_j \cap [[n_0+D-L, n_0+D-1]]) - (D,0) \\
&=A_j \cap [[n_0-L, n_0-1]]
\end{align*}
and similarly
\begin{align*}
A'_j \cap [[n_0+D, n_0+D+L]] &= (A_j \cap [[n_0, n_0+L]]) + (D,0) \\
&= A_j \cap [[n_0+D, n_0+D+L]]
\end{align*}
and thus $A_j, A'_j$ in fact agree on a larger region:
\begin{equation}\label{aj-ajp-agree}
A'_j \cap [[n_0-L, n_0+D+L]] = A_j \cap [[n_0-L, n_0+D+L]].
\end{equation}

It will now suffice to show that  $A'_1, \dots, A'_J$ solve the tiling equation 
$$\Tile(F_1,\dots,F_J; E),$$
that is to say that
$$ A'_1 \oplus F_1 \uplus \dots \uplus A'_J \oplus F_J = E.$$
Since both sides of this equation are periodic with respect to translations by $D\Z \times \{0\}$, it suffices to establish this claim within $[[n_0,n_0+D-1]]$, that is to say
\begin{equation}\label{apfe}
\biguplus_{j=1}^J \left((A'_j \oplus F_j) \cap [[n_0,n_0+D-1]]\right)
= E \cap [[n_0,n_0+D-1]].
\end{equation}
However, since $F_1,\dots,F_J$ are contained in $[[-L,L]]$, so the only portions of $A'_1,\dots,A'_J$ that are relevant for \eqref{apfe} are those in $[[n_0-L,n_0+D+L-1]]$.  But from \eqref{aj-ajp-agree} we may replace each $A'_j$ in \eqref{apfe} by $A_j$.  Since $A_1,\dots,A_J$ solve the tiling equation $\Tile(F_1,\dots,F_J; E)$, the claim follows.
\end{proof}

\begin{remark} An inspection of the argument reveals that the hypothesis that $G_0$ was abelian was not used anywhere in the proof, thus Theorem \ref{oned} is also valid for nonabelian $G_0$ (with suitable extensions to the notation).  This generalization will be used in Section \ref{nonab-sec}.
\end{remark}

\section{Combining multiple tiling equations into a single equation}\label{combine-sec}

In this section we establish Theorem \ref{combine}. For the rest of the section we use the notation and hypotheses of that theorem. 

\begin{remark} The reader may wish to first consider the special case $M=2, J=1, N=3$ in what follows to simplify the notation.  In this case, part (ii) of the theorem asserts that the system of tiling equations
$$ \Tile(F^{(1)}, \Z^d \times E_0^{(1)}),  \Tile(F^{(2)}, \Z^d \times E_0^{(2)})$$
in $\Z^d \times G_0$ is undecidable if and only if the single tiling equation
$$ \Tile( F^{(1)} \times \{1\} \uplus F^{(2)} \times \{2\}, \Z^d \times (E_0^{(1)} \times \{1\} \uplus E_0^{(2)} \times \{2\}))$$
in $\Z^d \times G_0 \times \Z_3$ is undecidable.
\end{remark}

We begin with part (ii).  Suppose we have a solution 
$$(A_1,\dots,A_J) \in \bigcap_{m=1}^M \Tile(F_1^{(m)},\dots,F_J^{(m)}; \Z^d \times E_0^{(m)})_\Universe$$
in $G$ to the system of tiling equations $\Tile(F_1^{(m)},\dots,F_J^{(m)}; \Z^d \times E_0^{(m)}), m=1,\dots,M$, thus
\begin{equation}\label{afm}
A_1 \oplus F_1^{(m)} \uplus \dots \uplus A_J \oplus F_J^{(m)} = \Z^d \times E_0^{(m)}
\end{equation}
for all $m=1,\dots,M$.  If we then define the sets
$$ \tilde A_j \coloneqq A_j \times \{0\} \subset G \times \Z_N$$
for $j=1,\dots,J$, then from construction of the $\tilde F_j$ we have
$$ \tilde A_j \oplus \tilde F_j = \biguplus_{m=1}^M (A_j \oplus F_j^{(m)}) \times \{m\}$$
for any $j=1,\dots,J$ and $m=1,\dots,M$, and hence by \eqref{afm}
$$ \tilde A_1 \oplus \tilde F_1 \uplus \dots \uplus \tilde A_J \oplus \tilde F_J^{(m)} = \biguplus_{m=1}^M (\Z^d \times E_0^{(m)}) \times \{m\}.$$
But by \eqref{e0-stack}, the right-hand side here is $\Z^d \times \tilde E_0$.  Thus we see that $\tilde A_1,\dots,\tilde A_J$ solve the single tiling equation 
$\Tile(\tilde F_1,\dots,\tilde F_J; \Z^d \times \tilde E_0)$.

Conversely, suppose that we have a solution
$$(\tilde A_1,\dots,\tilde A_J) \in \Tile(\tilde F_1,\dots,\tilde F_J; \Z^d \times \tilde E_0)_\Universe$$
in $G \times \Z_N$ to the tiling equation
$\Tile(\tilde F_1,\dots,\tilde F_J; \Z^d \times \tilde E_0)$; thus
\begin{equation}\label{ta1}
 \tilde A_1 \oplus \tilde F_1 \uplus \dots \uplus \tilde A_J \oplus \tilde F_J^{(m)} = \Z^d \times \tilde E_0.
 \end{equation}
 We claim that $\tilde A_j \subset G \times \{0\}$ for all $j=1,\dots,J$.  For if this were not the case, then there would exist $j=1,\dots,J$ and an element $(g,n)$ of $\tilde A_j$ with $n \in \Z_N \backslash \{0\}$.  On the other hand, for any $1 \leq m \leq M$, the set $F_j^{(m)}$ is non-empty, hence $\tilde F_j$ contains an element of the form $(f_m,m)$ for some $f_m \in G$.  By \eqref{ta1}, we then have $(g+f_m, n+m) \in \Z^d \times \tilde E_0$, hence by construction of $\tilde E_0$ we have
$$ n + m \in \{1,\dots,M\}$$
for all $m=1,\dots,M$, or equivalently
$$ n + \{1,\dots,M\} \subset \{1,\dots,M\}.$$
But since $N > M$, this is inconsistent with $n$ being a non-zero element of $\Z_N$.  Thus we have $\tilde A_j \subset G \times \{0\}$ as desired, and we may write
$$ \tilde A_j = A_j \times \{0\}$$
for some $A_j \subset G$.  By considering the intersection (or ``slice'') of \eqref{ta1} with $G \times \{m\}$, we see that
$$
A_1 \oplus F^{(m)}_1 \uplus \dots \uplus A_J \oplus F^{(m)}_J = \Z^d \times E^{(m)}_0
$$
for all $m=1,\dots,M$, that is to say $A_1,\dots,A_J$ solves the system
of tiling equations $\Tile(F_1^{(m)},\dots,F_J^{(m)}; \Z^d \times E_0^{(m)}), m=1,\dots,M$.  We have thus demonstrated that the equation $\Tile(\tilde F_1,\dots,\tilde F_J; \Z^d \times \tilde E_0)$ admits a solution if and only if the system $\Tile(F_1^{(m)},\dots,F_J^{(m)}; \Z^d \times E_0^{(m)}), m=1,\dots,M$ does.  This argument is also valid in any other universe $\Universe^*$ of ZFC, which gives (ii).  An inspection of the argument also reveals that the equation $\Tile(\tilde F_1,\dots,\tilde F_J; \Z^d \times \tilde E_0)$ admits a periodic solution if and only if the system $\Tile(F_1^{(m)},\dots,F_J^{(m)}; \Z^d \times E_0^{(m)}), m=1,\dots,M$ does, which gives (i).

As noted in the introduction, in view of Theorem \ref{combine} we see that to prove Theorem \ref{main} it suffices to prove Theorem \ref{main-red}.  This is the objective of the next five sections of the paper.

\begin{remark}\label{nonab} For future reference we remark that the abelian nature of $G_0$ was not used in the above argument, thus Theorem \ref{combine} is also valid for nonabelian $G_0$ (with suitable extensions to the notation).
\end{remark}

\section{From tiling to functions}\label{function-sec}

In this section we reduce Theorem \ref{main-red} to Theorem \ref{main-function}, by means of the following general proposition.

\begin{proposition}[Equivalence of tiling equations and functional equations]\label{tile-function}  Let $G$ be an explicit finitely generated abelian group, let $G_1$ be an explicit finite abelian group, let $J, M \geq 1$ and $N > J$ be standard natural numbers, and suppose that for each $j=1,\dots,J$ and $m=1,\dots,M$ one is given a (possibly empty) finite subset $H_j^{(m)}$ of $G$ and a (possibly empty) subset $F_j^{(m)}$ of $G_1$.  For each $m=1,\dots,M$, assume also that we are given a subset $E^{(m)}$ of $G_1$.  We adopt the abbreviations
\begin{align*}
    [[a]] &\coloneqq \{a\} \times G_1 \subset \Z_N \times G_1 \\
    [[a,b]] &\coloneqq \{ n \in \Z: a\leq n \leq b\} \times G_1\subset \Z_N\times G_1
\end{align*} 
for integers $a \leq b$.  Then the following are equivalent:
\begin{itemize}
\item[(i)]  The system of tiling equations
\begin{equation}\label{tile-1}
\Tile\left( \left( -H_j^{(m)} \times \{0\} \times F_j^{(m)} \uplus \{0\} \times [[j]])\right)_{j=1}^J; G \times (\{0\} \times E^{(m)} \uplus [[1,J]]) \right)
\end{equation}
for all $m=1,\dots,M$, together with the tiling equations
\begin{equation}\label{tile-2}
\Tile\left( \left( \{0\} \times [[\sigma(j)]] \right)_{j=1}^J; G \times [[1,J]] \right)
\end{equation}
for every permutation $\sigma \colon \{1,\dots,J\} \to \{1,\dots,J\}$ of $\{1,\dots,J\}$, admit a solution.
\item[(ii)]  There exist $f_j \colon G \to G_1$ for $j=1,\dots,J$ that obey the system of functional equations
\begin{equation}\label{func-eq-stack}
\biguplus_{j=1}^J \biguplus_{h_j \in H_j^{(m)}} (F_j^{(m)} + f_j(n+h_j)) = E^{(m)} 
\end{equation}
for all $n \in G$ and $m=1,\dots,M$.
\end{itemize}
\end{proposition}

\begin{remark} The reason why we work with $\{0\} \times F_j^{(m)} \uplus [[j]]$ instead of just $\{0\} \times F_j^{(m)}$ in \eqref{tile-1} is in order to ensure that one is working with a non-empty tile (as is required in Theorem \ref{main-red}), even when the original tile $F_j^{(m)}$ is empty.
\end{remark}

\begin{remark}
The reader may wish to first consider the special case $M=J=1$, $N=2$ in what follows to simplify the notation.  In this case, the theorem asserts that for any finite $H \subset G$, and $F, E \subset G_1$, the system of tiling equations
\begin{align*}
    A \oplus \left(  (-H \times \{0\} \times F) \uplus \{0\} \times \{1\} \times G_1)\right) &= G \times (\{0\} \times E \uplus \{1\} \times G_1) \\
A \oplus \{0\} \times \{1\} \times G_1 &= G \times \{1\} \times G_1
\end{align*} 
admits a solution $A \subset G \times \Z_2 \times G_1$ if and only if there is a function $f \colon G \to G_1$ obeying the equation
$$ \biguplus_{h \in H} (F + f(n+h)) = E$$
for all $n \in G$.  The relationship between the set $A$ and the function $f$ will be given by the graphing relation
$$ A = \{ (n,0,f(n)): n \in G\}.$$
\end{remark}

\begin{proof}  Let us first show that (ii) implies (i).  If $f_1,\dots,f_J$ obey the system \eqref{func-eq-stack}, we define the sets $A_1,\dots,A_J \subset G \times \Z_N \times G_1$ to be the graphs of $f_1,\dots,f_J$ in the sense that
\begin{equation}\label{aj-form}
 A_j \coloneqq \{ (n, 0, f_j(n)): n \in G \}.
\end{equation}

For any $j=1,\dots,J$ and permutation $\sigma \colon \{1,\dots,J\} \to \{1,\dots,J\}$, we have
\begin{equation}\label{j0}
A_j \oplus \{0\} \times [[\sigma(j)]] = G \times [[\sigma(j)]]
\end{equation}
which gives the tiling equation \eqref{tile-2} for any permutation $\sigma$.  Next, for $j=1,\dots,J$ and $m=1,\dots,M$, we have
\begin{equation}\label{aj-sum}
A_j \oplus -H_j^{(m)} \times \{0\} \times F_j^{(m)} = \biguplus_{n \in G} \{n\} \times \biguplus_{h_j \in H_j^{(m)}} \{0\} \times (F_j^{(m)} + f_j(n+h_j))
\end{equation}
and (as a special case of \eqref{j0})
$$A_j \oplus \{0\} \times [[j]] = G \times [[j]] 
$$
so that the tiling equation \eqref{tile-1} then follows from \eqref{func-eq-stack}.  This shows that (ii) implies (i).

Now assume conversely that (i) holds, thus we have sets $A_1,\dots,A_J \subset G \times \Z_N \times G_1$ obeying the system of tiling equations 
\begin{equation}\label{tile-1-alt}
\biguplus_{j=1}^J A_j \oplus \left( -H_j^{(m)} \times \{0\} \times F_j^{(m)} \uplus \{0\} \times [[j]])\right) = G \times (\{0\} \times E^{(m)} \uplus [[1,J]])
\end{equation}
for all $m=1,\dots,M$, and
\begin{equation}\label{tile-2-alt}
\biguplus_{j=1}^J A_j \oplus \{0\} \times [[\sigma(j)]]
= G \times [[1,J]] 
\end{equation}
for all permutations $\sigma \colon \{1,\dots,J\} \to \{1,\dots,J\}$.

We first adapt an argument from Section \ref{combine-sec} to claim that each $A_j$ is contained in $G \times [[0]]$.  For if this were not the case, there would exist $j=1,\dots,J$ and an element $(g,n,g_0)$ of $A_j$ with $n \in \Z_N \backslash \{0\}$.  The left-hand side of the tiling equation \eqref{tile-2-alt} would then contain $(g,n+\sigma(j),g_0)$, and thus we would have
$$ n + \sigma(j) \in \{1,\dots,J\}$$
for all permutations $\sigma$, thus
$$ n + \{1,\dots,J\} \subset \{1,\dots,J\}.$$
But this is inconsistent with $n$ being a non-zero element of $\Z_N$.  Thus each $A_j$ is contained in $G \times [[0]]$ as claimed. 

If one considers the intersection (or ``slice'') of \eqref{tile-2-alt} with $G \times [[\sigma(j)]]$, we conclude that
$$ A_j \oplus \{0\} \times [[\sigma(j)]] = G \times [[ \sigma(j) ]]$$
for any $j=1,\dots,J$ and permutation $\sigma$.  This implies that for each $n \in G$ there is a unique $f_j(n) \in G_1$ such that $(n,0,f_j(n)) \in A_j$, thus the $A_j$ are of the form \eqref{aj-form} for some functions $f_j$.  The identity \eqref{aj-sum} then holds, and so from inspecting the $G \times [[0]]$ ``slice'' of \eqref{tile-1-alt} we obtain the equation \eqref{func-eq-stack}.  This shows that (ii) implies (i).
\end{proof}

The proof of Proposition \ref{tile-function} is valid in every universe $\Universe^*$ of ZFC, thus the solvability question in Proposition \ref{tile-function}(i) is decidable if and only if the solvability question in Proposition \ref{tile-function}(ii) is.  Applying this fact for $J=2$, we see that  Theorem \ref{main-function} implies Theorem \ref{main-red}.

It now remains to establish Theorem \ref{main-function}.  This is the objective of the next four sections of the paper.

\section{Reduction to the Hamming cube}\label{hamming-sec}

In this section we show how Theorem \ref{main-hamming} implies Theorem \ref{main-function}. Let $N$, $D$, $M$, $h_1^{(m)}$, $h_2^{(m)}$, $F_1^{(m)}$, $F_2^{(m)}$, $E^{(m)}$ be as in Theorem \ref{main-hamming}.  For each $d=1,\dots,D$, let $\pi_d \colon \Z_N^D \to \Z_N$ denote the $d^{\mathrm{th}}$ coordinate projection, thus
\begin{equation}\label{coord}
y = (\pi_1(y),\dots,\pi_D(y))
\end{equation}
for all $y \in \Z_N^D$.

We write elements of $\Z^2 \times \Z_2$ as $(n,t)$ with $n \in \Z^2$ and $t \in \Z_2$.  For a pair of functions $\tilde f_1,\tilde f_2 \colon \Z^2 \times \Z_2 \to \Z_N^D$, consider the system of functional equations
\begin{equation}\label{sys-1}
\left(\pi_d^{-1}(\{0\}) + \tilde f_j(n,t)\right) \uplus 
\left(\pi_d^{-1}(\{0\}) + \tilde f_j(n,t+1)\right) = \pi_d^{-1}(\{-1,1\}) 
\end{equation}
for $(n,t) \in \Z^2 \times \Z_2$, $d=1,\dots,D$ and $j=1,2$, as well as the equations
\begin{equation}\label{sys-2}
\left(F_1^{(m)} + \tilde f_1((n,t)+(h_1^{(m)},0))\right) \uplus
\left(F_2^{(m)} + \tilde f_2((n,t)+(h_2^{(m)},0))\right) = E^{(m)}
\end{equation}
for $(n,t) \in \Z^2 \times \Z_2$ and $m=1,\dots,M$.  Note that this system is of the form \eqref{func-eq} (with $f_j$ replaced by $\tilde f_j$, and for suitable choices of $M$, $F_1^{(m)}, F_2^{(m)}$, $E^{(m)}$).  It will therefore suffice to establish (using an argument formalizable in ZFC) the equivalence of the following two claims:

\begin{itemize}
\item[(i)]  There exist functions $\tilde f_1, \tilde f_2 \colon \Z^2 \times \Z_2 \to \Z_N^D$ solving the system \eqref{sys-1}, \eqref{sys-2}.
\item[(ii)] There exist $f_1,f_2 \colon \Z^2 \to \{-1,1\}^D$ solving the system \eqref{func-eq-1}.
 \end{itemize}
 
 \begin{remark} As a simplified version of this equivalence, the reader may wish to take $M=1$, $D=2$, and only work with a single function $f$ (or $\tilde f$) instead of a pair $f_1,f_2$ (or $\tilde f_1,\tilde f_2$) of functions.  The claim is then that the following two statements are equivalent for any $F,E \subset \Z_N^2$:
 \begin{itemize}
     \item [(i')]  There exists $\tilde f \colon \Z^2 \times \Z^2 \to \Z_N^2$ obeying the equations:
         \[(\{0\} \times \Z_N + \tilde f(n,t)) \uplus (\{0\} \times \Z_N + \tilde f(n,t+1)) = \{-1,1\} \times \Z_N,\]
         \[(\Z_N \times \{0\} + \tilde f(n,t)) \uplus (\Z_N \times \{0\} + \tilde f(n,t+1)) = \Z_N \times \{-1,1\},\] and \[F + \tilde f(n,t) = E \] for all $(n,t) \in \Z^2 \times \Z_2$.
     \item[(ii')]  There exists $f \colon \Z^2 \to \{-1,1\}^2$ obeying the equation $F + \eps f(n) = E$ for all $n \in \Z^2$ and $\eps = \pm 1$.
 \end{itemize}
 The relation between (i') and (ii') shall basically arise from the ansatz $\tilde f(n,t) = (-1)^t f(n)$.
 \end{remark}
 
We first show that (ii) implies (i).
Given solutions $f_1,f_2$ to the system \eqref{func-eq-1}, we define the functions $\tilde f_1, \tilde f_2 \colon \Z^2 \times \Z_2 \to \Z_N^D$ by the formula
$$ \tilde f_j(n,t) \coloneqq (-1)^t f_j(n)$$
for $j=1,2$, $n \in \Z^2$, and $t \in \Z_2$, with the conventions $(-1)^0 \coloneqq 1$ and $(-1)^1 \coloneqq -1$.  The equations \eqref{func-eq-1} then imply \eqref{sys-2}, while the fact that the $f_j$ takes values in $\{-1,1\}^D$ implies \eqref{sys-1} (the key point here is that $\{-1,1\} = \{x\} \uplus \{-x\}$ if and only if $x \in \{-1,1\}$).  This proves that (ii) implies (i).

Now we prove (i) implies (ii).  Let $\tilde f_1, \tilde f_2 \colon \Z^2 \times \Z_2 \to \Z_N^D$ be solutions to \eqref{sys-1}, \eqref{sys-2}.  From \eqref{sys-1} we see (on applying the projection $\pi_d$) that
$$ \{ \pi_d(\tilde f_j(n,t)) \} \uplus
\{ \pi_d(\tilde f_j(n,t+1)) \} = \{-1,1\}$$
for all $j=1,2$, $d=1,\dots,D$, and $(n,t) \in \Z^2 \times \Z_2$, or equivalently that
$$ \pi_d(\tilde f_j(n,t)) \in \{-1,1\}$$
and
$$ \pi_d(\tilde f_j(n,t+1)) = - \pi_d(\tilde f_j(n,t))$$
for all $j=1,2$, $d=1,\dots,D$, and $(n,t) \in \Z^2 \times \Z_2$. 
From \eqref{coord}, we thus have
$$ \tilde f_j(n,t) \in \{-1,1\}^D$$
and
$$ \tilde f_j(n,t+1) = - \tilde f_j(n,t)$$
for all $j=1,2$ and $(n,t) \in \Z^2 \times \Z_2$.  Thus we may write
$$ \tilde f_j(n,t) = (-1)^t f_j(n)$$
for some functions $f_j \colon \Z^2 \to \{-1,1\}^D$.  The system \eqref{sys-2} is then equivalent to the system of equations \eqref{func-eq-1}.  This shows that (i) implies (ii).  These arguments are valid in every universe $\Universe^*$ of ZFC, thus Theorem \ref{main-hamming} implies Theorem \ref{main-function}.

It now remains to establish Theorem \ref{main-hamming}.  This is the objective of the next three sections of the paper.

\section{Reduction to systems of linear equations on boolean functions}\label{linear-sec}

In this section we show how Theorem \ref{main-linear} implies Theorem \ref{main-hamming}. Let $D$, $D_0$, $M_1$, $M_2$, $a_{j,d}^{(m)}$, $h_d$ be as in Theorem \ref{main-linear}.  We let $N$ be a sufficiently large integer.  For each $j=1,2$ and $m = 1,\dots,M_j$, we consider the subgroup $H_j^{(m)}$ of $\Z_N^D$ defined by
\begin{equation}\label{hjm-def}
H_j^{(m)} \coloneqq \left\{ (y_1,\dots,y_D) \in \Z_N^D \colon \sum_{d=1}^D a_{j,d}^{(m)} y_j = 0 \right\}
\end{equation}
and let $\pi_d \colon \Z_N^D \to \Z_N$ for $d=1,\dots,D$ be the coordinate projections as in the previous section.
For some unknown functions $f_1,f_2 \colon \Z^2 \to \{-1,1\}^D \subset \Z_N^D$ we consider the system of functional equations
\begin{equation}\label{system-1}
H_j^{(m)} + \epsilon f_j(n) = H_j^{(m)}
\end{equation}
for all $n \in \Z^2$, $j=1,2$, $m=1,\dots,M_j$, and $\epsilon = \pm 1$, as well as the system
\begin{equation}\label{system-2}
(\pi_d^{-1}(\{0\}) + \epsilon f_1(n)) \uplus (\pi_d^{-1}(\{0\}) + \epsilon f_2(n+h_d)) = \pi_d^{-1}(\{-1,1\})
\end{equation}
for all $n \in \Z^2$, $d=1,\dots,D_0$ and $\epsilon = \pm 1$.  Note that this system \eqref{system-1}, \eqref{system-2} is of the form required for Theorem \ref{main-hamming}.  It will suffice to establish (using an argument valid in every universe of ZFC) the equivalence of the following two claims:

\begin{itemize}
\item[(i)]  There exist functions $f_1, f_2 \colon \Z^2 \to \Z_N^D$ solving the system \eqref{system-1}, \eqref{system-2}.
\item[(ii)]  There exist functions $f_{j,d} \colon \Z^2 \to \{-1,1\}$ solving the system \eqref{func-eq-2}, \eqref{func-eq-3}.
\end{itemize}

\begin{remark} To understand this equivalence, the reader may wish to begin by verifying two simplified special cases of this equivalence.  Firstly, the two (trivially true) statements
\begin{itemize}
    \item[(i')]  There exist a function $f \colon \Z^2 \to \{-1,1\}^2$ solving the equation $$H + \epsilon f(n) = H$$ for all $n \in \Z^2$ and $\epsilon = \pm 1$, where $H \coloneqq \{ (y_1,y_2) \in \Z_N^2: y_1+y_2=0\}$.
    \item[(ii')]  There exist functions $f_1,f_2 \colon \Z^2 \to \{-1,1\}$ solving the equation $$f_1(n)+f_2(n) = 0$$ for all $n \in \Z^2$.
\end{itemize}
can be easily seen to be equivalent after making the substitution $f(n) = (f_1(n), f_2(n))$.  Secondly, for any $h \in \Z^2$, the two (trivially true) statements
\begin{itemize}
    \item[(i'')]  There exist a functions $f_1,f_2 \colon \Z^2 \to \{-1,1\}$ solving the equation $$(\{0\} + \epsilon f_1(n)) \uplus (\{0\} + \epsilon f_2(n+h)) = \{-1,1\}$$ for all $n \in \Z^2$ and $\epsilon = \pm 1$.
    \item[(ii'')]  There exist functions $f_1,f_2 \colon \Z^2 \to \{-1,1\}$ solving the equation $$f_2(n+h) = -f_1(n)$$ for all $n \in \Z^2$.
\end{itemize}
are also easily seen to be equivalent (the solution sets $(f_1,f_2)$ for (i'') and (ii'') are identical). 
\end{remark}

Returning to the general case, we first show that (ii) implies (i).  Let  $f_{j,d} \colon \Z^2 \to \{-1,1\}$ be solutions to \eqref{func-eq-2}, \eqref{func-eq-3}.  We let $f_j \colon \Z^2 \to \{-1,1\}^D$ be the function
\begin{equation}\label{fjn}
 f_j(n) \coloneqq (f_{j,1}(n), \dots, f_{j,D}(n))
 \end{equation}
for $n \in \Z^2$ and $j=1,2$, where we now view the Hamming cube $\{-1,1\}^D$ as lying in $\Z_N^D$.  For any $j=1,2$, $m=1,\dots,M$, $n \in \Z^2$, and $\epsilon= \pm 1$, we see from \eqref{func-eq-2}, \eqref{hjm-def} that
$$ \epsilon f_j(n) \in H_j^{(m)}$$
and hence \eqref{system-1} holds.  Similarly, for any $d=1,\dots,D_0$, $n \in \Z^2$, and $\epsilon = \pm 1$ we have from \eqref{func-eq-3} that
$$ (\{0\} + \epsilon f_{1,d}(n)) \uplus (\{0\} + \epsilon f_{2,d}(n+h_d)) = \{-1,1\}$$
which implies \eqref{system-2}. This shows that (ii) implies (i).

Now we show that (i) implies (ii).  Let $f_1,f_2$ be a solution to the system \eqref{system-1}, \eqref{system-2}.  We may express $f_j$ in components as
\eqref{fjn}, where the $f_{j,d}$ are functions from $\Z^2$ to $\{-1,1\}$.  From \eqref{system-1} we see that
$$ (f_{j,1}(n),\dots,f_{j,D}(n)) \in H_j^{(m)}$$
for all $n \in \Z^d$, $j=1,2$, $m=1,\dots,M_j$ (viewing the tuple as an element of $\Z_N^D$), or equivalently that
$$ \sum_{d=1}^D a_{j,d}^{(m)} f_{j,d}(n) = 0 \hbox{ mod } N.$$
The left-hand side is an integer that does not exceed $\sum_{d=1}^D |a_{j,d}^{(m)}|$ in magnitude, so for $N$ large enough we have
$$ \sum_{d=1}^D a_{j,d}^{(m)} f_{j,d}(n) = 0,$$
that is to say \eqref{func-eq-1} holds. Similarly, from \eqref{system-2} we see that
$$ \{ f_{1,d}(n)\} \uplus \{ f_{2,d}(n + h_d) \} = \{-1,1\}$$
for all $n \in \Z^2$ and $d=1,\dots,D_0$, which gives \eqref{func-eq-2}.  This proves that (i) implies (ii). These arguments are valid in every universe of ZFC, thus Theorem \ref{main-linear} implies Theorem \ref{main-hamming}.

It now remains to establish Theorem \ref{main-linear}.  This is the objective of the next two sections of the paper.

\section{Reduction to a local Boolean constraint}\label{boolean-sec}

In this section we show how Theorem \ref{main-boolean} implies Theorem \ref{main-linear}.  (One can also easily establish the converse implication, but we will not need that implication here.)

We begin with some preliminary reductions.  We first claim that Theorem \ref{main-boolean} implies a strengthening of itself in which the set $\Omega$ can be taken to be symmetric: $\Omega = -\Omega$; also, we can take $D \geq 2$.  To see this, suppose that we can find $D,L,h_1,\dots,h_L,\Omega$ obeying the conclusions of Theorem \ref{main-boolean}.  We then introduce the symmetric set $\Omega' \subset \{-1,1\}^{(D+1)L}$ to be the collection of all tuples $(\omega_{d,l})_{d=1,\dots,D+1; l=1,\dots,L}$ obeying the constraints
$$
\omega_{D+1,1} = \dots = \omega_{D+1,L}$$
and
$$ (\omega_{d,l} \omega_{D+1,l})_{d=1,\dots,D;l=1\dots,L} \in \Omega.$$
Clearly $\Omega'$ is symmetric.  If $f_1,\dots,f_D \colon \Z^2 \to \{-1,1\}$ obeys the constraint \eqref{boolean}, then by setting $f_{D+1} \colon \Z^2 \to \{-1,1\}$ to be the constant function $f_{D+1}(n)=1$, we see from construction that
\begin{equation}\label{fad}
(f_d(n+h_l))_{d=1,\dots,D+1; l=1,\dots,L} \in \Omega'
\end{equation}
for all $n \in \Z^2$.  Conversely, if there was a solution $f_1,\dots,f_{D+1} \colon \Z^2 \to \{-1,1\}$ to \eqref{fad}, then we must have
$$ f_{D+1}(n+h_1) = \dots = f_{D+1}(n+h_L)$$
and
$$ (f_d(n+h_l) f_{D+1}(n+h_l))_{d=1,\dots,D;l=1\dots,L} \in \Omega,$$
and then the functions $f_d f_{D+1} \colon \Z^2 \to \{-1,1\}$ for $d=1,\dots,D$ would form a  solution to \eqref{boolean}.  As these arguments are formalizable in ZFC, we see that Theorem \ref{main-boolean} for the specified choice of $D,L,h_1,\dots,h_L,\Omega$ implies Theorem \ref{main-boolean} for $D+1,L,h_1,\dots,h_L,\Omega'$, giving the claim.

Now consider the system \eqref{boolean} for some $D,L,h_1,\dots,h_L,\Omega$ with $D \geq 2$ and $\Omega \subset \{-1,1\}^{DL}$ symmetric, and some unknown functions $f_1,\dots,f_D \colon \Z^2 \to \{-1,1\}$.  The constraint \eqref{boolean} involves multiple functions as well as multiple shifts.  We now ``decouple'' this constraint into an equivalent system of simpler constraints, which either involve just two functions, or do not involve any shifts at all.  Namely, we introduce a variant system involving some other unknown functions $f_{j,d,l} \colon \Z^2 \to \{-1,1\}$ with $j=1,2$, $d=1,\dots,D$, $l=1,\dots,L$, consisting of the symmetric Boolean constraint
\begin{equation}\label{bool-1}
(f_{1,d,l}(n))_{d=1,\dots,D; l=1,\dots,L} \in \Omega,
\end{equation}
for all $n \in \Z^2$, the additional symmetric Boolean constraints
\begin{equation}\label{bool-2}
f_{2,d,1}(n) = \dots = f_{2,d,L}(n)
\end{equation}
for all $n \in \Z^2$ and $d=1,\dots,D$, and the shifted constraints
\begin{equation}\label{bool-3}
f_{2,d,l}(n+h_l) = - f_{1,d,l}(n)
\end{equation}
for all $n \in \Z^2$, $d=1,\dots,D$, and $l=1,\dots,L$.  Observe that if $f_1,\dots,f_D \colon \Z^2 \to \{-1,1\}$ solve \eqref{boolean}, then the functions $f_{j,d,l} \colon\Z^2 \to \{-1,1\}$ defined by
$$ f_{1,d,l}(n) \coloneqq f_d(n+h_l)$$
and
$$ f_{2,d,l}(n) \coloneqq - f_d(n)$$
obey the system \eqref{bool-1}, \eqref{bool-2}, \eqref{bool-3}; conversely, if $f_{j,d,l} \colon \Z^2 \to \{-1,1\}$ solve \eqref{bool-1}, \eqref{bool-2}, \eqref{bool-3}, then from \eqref{bool-2} we have $f_{2,d,l}(n) = -f_d(n)$ for all $d=1,\dots,D$, $l=1,\dots,L$, and some functions $f_1,\dots,f_D \colon \Z^2 \to \{-1,1\}$, and then from \eqref{bool-1}, \eqref{bool-3} we see that $f_1,\dots,f_D$ solve \eqref{boolean}.  These arguments are formalizable in ZFC, so we conclude that the question of whether the system \eqref{bool-1}, \eqref{bool-2}, \eqref{bool-3} admits solutions is undecidable.

A symmetric set $\Omega \subset \{-1,1\}^{DL}$ can be viewed as the Hamming cube $\{-1,1\}^{DL}$ with a finite number of pairs of antipodal points $\{\epsilon, -\epsilon\}$ removed.  The constraint \eqref{bool-2} is constraining the tuple $(f_{2,d,l}(n))_{d=1,\dots,D; l=1,\dots,L}$ to a symmetric subset of $\{-1,1\}^{DL}$, which can thus also be viewed in this fashion.  Relabeling $f_{j,d,l}$ as $f_{j,d}$ for $d=1,\dots,D_0 \coloneqq DL$, and assigning the shifts $h_1,\dots,h_L$ to these labels appropriately, we conclude the following consequence of Theorem \ref{main-boolean}:

\begin{theorem}[An undecidable system of antipode-avoiding constraints]\label{antipode}  There exist standard integers $D_0 \geq 2$ and $M_1,M_2 \geq 1$, shifts $h_1,\dots,h_{D_0} \in \Z^2$, and vectors $\epsilon_j^{(m)} \in \{-1,1\}^{D}$ for $j=1,2$ and $m=1,\dots,M_j$ such that the question of whether there exist functions $f_{j,d} \colon \Z^2 \to \{-1,1\}$, for $j=1,2$, $d=1,\dots,D_0$ that solve the constraints
\begin{equation}\label{antipode-1}
(f_{j,d}(n))_{d=1,\dots,D_0} \not \in \{ -\epsilon_j^{(m)}, \epsilon_j^{(m)} \}
\end{equation}
for all $n \in \Z^2$, $j=1,2$, $m =1,\dots,M_j$, as well as the constraints
\begin{equation}\label{antipode-2}
f_{2,d}(n+h_d) = - f_{1,d}(n)
\end{equation}
for all $n \in \Z^2$ and  $d = 1,\dots,D_0$, is undecidable (when expressed as a first-order sentence in ZFC).
\end{theorem}

This is already quite close to Theorem \ref{main-linear}, except that the linear constraints \eqref{func-eq-2} have been replaced by antipode-avoiding constraints \eqref{antipode-1}.  To conclude the proof of Theorem \ref{main-linear}, we will show that each antipode-avoiding constraint \eqref{antipode-1} can be encoded as a linear constraint of the form \eqref{func-eq-2} after adding some more functions.

To simplify the notation we will assume that $M_1=M_2=M$, which one can assume without loss of generality by repeating the vectors $\epsilon_j^{(m)}$ as necessary.  The key observation is the following.  If $\epsilon = (\epsilon_1,\dots,\epsilon_{D_0}) \in \{-1,1\}^{D_0}$ and $y_1,\dots,y_{D_0} \in \{-1,1\}^{D_0}$, then the following claims are equivalent:
\begin{itemize}
\item[(a)] $(y_1,\dots,y_{D_0}) \not \in \{-\epsilon, \epsilon\}$.
\item[(b)] $\epsilon_1 y_1 + \dots + \epsilon_{D_0} y_{D_0} \in \{ -D_0 + 2, -D_0 + 4, \dots, D_0 - 4, D_0 - 2 \}$.
\item[(c)] There exist $y'_1,\dots,y'_{D_0-2} \in \{-1,1\}$ such that $\epsilon_1 y_1 + \dots + \epsilon_{D_0} y_{D_0} + y'_1 + \dots + y'_{D_0-2} = 0$.
\end{itemize}
Indeed, it is easy to see from the triangle inequality and parity considerations (and the hypothesis $D_0 \geq 2$) that (a) and (b) are equivalent, and that (b) and (c) are equivalent. The point is that the antipode-avoiding constraint (a) has been converted into a linear constraint (c) via the addition of some additional variables.

\begin{example}  As a simple example of this equivalence (with $D_0=4$ and $\epsilon_1=\dots=\epsilon_4=1$), given a triple $(y_1,y_2,y_3,y_4) \in \{-1,1\}^4$, we see that the following claims are equivalent:
\begin{itemize}
    \item [(a')] $(y_1,y_2,y_3,y_4) \not \in \{ (-1,-1,-1,-1), (+1,+1,+1,+1) \}$.
    \item[(b')] $y_1+y_2+y_3+y_4 \in \{-2,0,+2\}$.
    \item[(c')] There exist $y_5, y_6 \in \{-1,1\}$ such that $y_1+y_2+y_3+y_4+y_5+y_6 = 0$.
\end{itemize}
\end{example}

We now set $D \coloneqq D_0 + M(D_0-2)$ and consider the question of whether there exist functions $f_{j,d} \colon \Z^2 \to \{-1,1\}$, for $j=1,2$, $d=1,\dots,D$ that solve the linear system of equations
\begin{equation}\label{cons}
\sum_{d=1}^{D_0} \epsilon_{j,d}^{(m)} f_{j,d}(n) + \sum_{d=1}^{D_0-2} f_{j,D_0+(m-1)(D_0-2)+d}(n) = 0
\end{equation}
for $j=1,2$, $m=1,\dots,M$, $n \in \Z^2$, as well as the linear system \eqref{antipode} for $j=1,2$, $n \in \Z^2$, and $d=1,\dots,D_0$.  In view of the equivalence of (a) and (c) (and the fact that for each $j=1,2$, $m=1,\dots,M$, and $n \in \Z^2$, the variables $f_{j,D_0+(m-1)(D_0-2)+d}(n)$ appear in precisely one constraint, namely the equation \eqref{cons} for the indicated values of $j,m,n$) we see that this system of equations \eqref{antipode-2}, \eqref{antipode} admits a solution if and only if the system of equations \eqref{antipode-1}, \eqref{antipode-2} admits a solution.  This argument is valid in every universe of ZFC, hence the solvability of the system \eqref{antipode-2}, \eqref{antipode} is undecidable.  This completes the derivation of Theorem \ref{main-linear} from Theorem \ref{main-boolean}.

It now remains to establish Theorem \ref{main-boolean}.  This is the objective of the next section of the paper.

\section{Undecidability of local Boolean constraints}\label{tiling}

In this section we prove Theorem \ref{main-boolean}, which by the preceding reductions also establishes Theorem \ref{main}.

Our starting point is the existence of an undecidable tiling equation
\begin{equation}\label{tile-z2}
 \Tile(F_1,\dots,F_J; \Z^2)
\end{equation}
for some standard $J$ and some finite $F_1,\dots,F_J \subset \Z^2$.  This was first shown\footnote{\label{berfoot}Berger's construction is able to encode any instance of the halting problem for Turing machines (with empty input) as a (Wang) tiling problem; since the consistency of ZFC is an undecidable statement equivalent to the non-halting of a certain Turing machine with empty input, this gives the required claim of undecidability.} in \cite{Ber} (after applying the reduction in \cite{golomb}), with many subsequent proofs; see for instance \cite{jv} for a survey. One can for instance take the tile set in \cite{ollinger-2}, which has $J = 11$, though the exact value of $J$ will not be of importance here.  

Note that to any solution 
$$(A_1,\dots,A_J) \in \Tile(F_1,\dots,F_J;\Z^2)_\Universe$$
in $\Z^2$ of the tiling equation \eqref{tile-z2}, one can associate a coloring function $c \colon \Z^2 \to C$ taking values in the finite set
$$ C \coloneqq \biguplus_{j=1}^J \{j\} \times F_j$$
by defining
$$ c(a_j + h_j) \coloneqq (j, h_j)$$
whenever $j=1,\dots,J$, $a_j \in A_j$, and $h_j \in F_j$.  The tiling equation \eqref{tile-z2} ensures that the coloring function $c$ is well-defined.  Furthermore, from construction we see that $c$ obeys the constraint
\begin{equation}\label{cna}
c(n) = (j,h_j) \implies c(n-h_j+h'_j) = (j,h'_j)
\end{equation}
for all $n \in \Z^2$, $j=1,\dots,J$, and $h_j, h'_j \in F_j$.  Conversely, suppose that $c \colon \Z^2 \to C$ is a function obeying \eqref{cna}.  Then if we define $A_j$ for each $j=1,\dots,J$ to be the set of those $a_j \in \Z^2$ such that $c(a_j+h_j) = (j,h_j)$ for some $h_j \in F_j$, from \eqref{cna} we have $c(a_j+f'_j) = (j,f'_j)$ for all $j=1,\dots,J$, $a_j \in A_j$, and $f'_j \in F_j$, which implies that $A_1,\dots,A_J$ solve the tiling equation \eqref{tile-z2}.  Thus the solvability of \eqref{tile-z2} is equivalent to the solvability of the equation \eqref{cna}; as the former is undecidable in ZFC, the latter is also, since the above arguments are valid in every universe of ZFC.

Since the set $C=\biguplus_{j=1}^J \{j\} \times F_j$ is finite, one can establish an explicit bijection $\iota \colon C \to \Omega$ between this set and some subset $\Omega$ of $\{-1,1\}^D$ for some $D$.  Composing $c$ with this bijection, we see that the question of locating Boolean functions $f_1,\dots,f_D \colon \Z^2 \to \{-1,1\}$ obeying the constraints
\begin{equation}\label{con-1}
(f_1(n),\dots,f_D(n)) \in \Omega
\end{equation}
and
\begin{equation}\label{con-2}
(f_1(n),\dots,f_D(n)) = \iota(j,h_j) \implies (f_1(n-h_j+h'_j),\dots,f_D(n-h_j+h'_j)) = \iota(j,h'_j)
\end{equation}
for all $n \in \Z^2$, $j=1,\dots,J$, and $h_j, h'_j \in F_j$, is undecidable in ZFC.  However, this set of constraints is of the type considered in Theorem \ref{main-boolean} (after enumerating the set of differences $\{ h_j - h'_j: j=1,\dots,J; h_j,h'_j \in F_j\}$ as $h_1,\dots,h_L$ for some $L$, and combining the various constraints in \eqref{con-1}, \eqref{con-2}), and the claim follows.

\section{Proof of \thmref{main'} }\label{extension}

In this section we modify the ingredients of the proof of \thmref{main} to establish \thmref{main'}. The proofs of both theorems proceed along similar lines, and in fact are both deduced from a common result in Theorem \ref{main-hamming}; see Figure \ref{fig:logic}.

We begin by proving the following analogue of \thmref{combine}. 

\begin{theorem}[Combining multiple tiling equations into a single equation]\label{combine'}  Let $J,M,d \geq 1$ and $N>M$ be standard natural numbers.  
For each $m=1,\dots,M$, let $F_1^{(m)},\dots,F_J^{(m)}$ be finite non-empty subsets of $\Z^d$, and let $E^{(m)}$ be a periodic subset of $\Z^d$.  Define the finite sets $\tilde F_1,\dots,\tilde F_J \subset \Z^d \times \{1,\dots, M\}$ and the periodic set $\tilde E \subset \Z^d \times \Z$ by 
\begin{equation}\label{fj-stack'}
 \tilde F_j \coloneqq \biguplus_{m=1}^M F_j^{(m)} \times \{m\}
 \end{equation}
and
\begin{equation}\label{e0-stack'}
 \tilde E \coloneqq  \biguplus_{m=1}^M E^{(m)} \times (N\Z+m ).
 \end{equation}
\begin{itemize}
\item[(i)]  The system $\Tile(F_1^{(m)},\dots,F_J^{(m)}; E^{(m)}), m=1,\dots,M$ of tiling equations is aperiodic if and only if the tiling equation
$\Tile(\tilde F_1,\dots,\tilde F_J; \tilde E)$ is aperiodic.
\item[(ii)]  The system $\Tile(F_1^{(m)},\dots,F_J^{(m)}; E^{(m)}), m=1,\dots,M$ of tiling equations is undecidable if and only if the tiling equation
$\Tile(\tilde F_1,\dots,\tilde F_J; \tilde E)$ is undecidable.
\end{itemize}
\end{theorem}

\begin{proof}
We will just prove (i); the proof of (ii) is similar and is left to the reader.  

The argument will be a ``pullback'' of the corresponding proof of Theorem \ref{combine}(i). 
 First, suppose that the system $\Tile(F_1^{(m)},\dots,F_J^{(m)}; E^{(m)}), m=1,\dots,M$ of tiling equations has a periodic solution $A_1,\dots,A_J \subset \Z^d$, thus
 \begin{equation}\label{afaf} 
 A_1 \oplus F_1^{(m)} \uplus \dots A_J \oplus F_J^{(m)} = E^{(m)}
 \end{equation}
 for $m=1,\dots,M$.  If we then introduce the periodic sets 
 $$\tilde A_j \coloneqq A_j\times N\Z \subset \Z^d \times \Z,\quad j=1,\dots,J$$
 then we have
 $$ \tilde A_j \oplus \tilde F_j = \biguplus_{m=1}^M (A_j \oplus F_j^{(m)}) \times (N\Z+m)$$
 for all $j=1,\dots,J$, and hence by \eqref{afaf}, \eqref{e0-stack'} we have
 \begin{equation}\label{taf}
 \tilde A_1 \oplus \tilde F_1 \uplus \dots \tilde A_J \oplus \tilde F_j = \tilde E.
 \end{equation}
Thus we have a periodic solution for the system
 $\Tile(\tilde F_1,\dots,\tilde F_J; \tilde E)$.
 
 Conversely, suppose that the system $\Tile(\tilde F_1,\dots,\tilde F_J; \tilde E)$ admits a periodic solution $\tilde A_1,\dots,\tilde A_J$, so that \eqref{taf} holds. Observe that if  $\tilde A_j\subset \Z^d \times N\Z$ for each $j=1,\dots,J$, then the ``slices'':
 $$A_j \coloneqq \{a\in\Z^d: (a,0)\in\tilde A_j\},\quad j=1,\dots, J$$
would be periodic and obey the equation \eqref{afaf} for every $m=1,\dots,M$, thus giving a periodic solution to the system of tiling equations $$\Tile(F_1^{(m)},\dots,F_J^{(m)}; E^{(m)}), \quad  m=1,\dots,M.$$

Now, suppose to the contrary that there is  $j_0 = 1,\dots,J$  such that there exists $(g,u)\in \tilde A_{j_0}$ where $n\in\Z\setminus N\Z$. 
From \eqref{taf} we see that for every $(f,m)\in\tilde F_{j_0}^{(m)}\neq \emptyset$, we have 
$$(g+f,u+m)\in \tilde E.$$ Thus $u+m\in \{1,\dots,M\}\oplus N\Z$ for every  $m = 1,\dots,M$. This is only possible if $u\in N\Z$, a contradiction.
Therefore, we have $\tilde A_j\subset \Z^d\times N\Z$,  for every $j = 1,\dots,J$ as needed.
\end{proof}

As in the proof of Theorem \ref{main},   \thmref{combine'} allows one to reduce the proof of \thmref{main'} to proving the following statement.

\begin{theorem}[An undecidable system of tiling equations with two tiles in $\Z^d$]\label{main-red'}  There exist standard natural numbers $d,M$, and for each $m=1,\dots,M$ there exist finite non-empty sets $F_1^{(m)}, F_2^{(m)} \subset \Z^d$ and periodic sets $E^{(m)} \subset \Z^d$ such that the system of tiling equations $\Tile(F^{(m)}_1,F^{(m)}_2; E^{(m)}), m=1,\dots,M$ is undecidable.
\end{theorem}

We will  show that  \thmref{main-hamming}  implies \thmref{main-red'}.  In order for the arguments from Section \ref{function-sec} to be effectively pulled back, we will first need to construct a rigid tile that can encode a finite group $\Z^k/\Lambda$ as the solution set to a tiling equation.

\begin{lemma}[A rigid tile]\label{rigid-tile}  Let $N_1,\dots,N_k \geq 5$, and let $\Lambda \leq \Z^k$ be the lattice
$$ \Lambda \coloneqq N_1\Z \times \dots \times N_k \Z.$$
Then there exists a finite subset $R$ of $\Z^k$ with the property that the solution set $\Tile(R;\Z^k)_\Universe$ of the tiling equation $\Tile(R;\Z^k)$ consists precisely of the cosets  $h+\Lambda$ of $\Lambda$, that is to say
$$\Tile(R;\Z^k)_\Universe = \Z^k / \Lambda.$$
\end{lemma}

\begin{figure}
    \centering
\centerline{\includegraphics[height=3in]{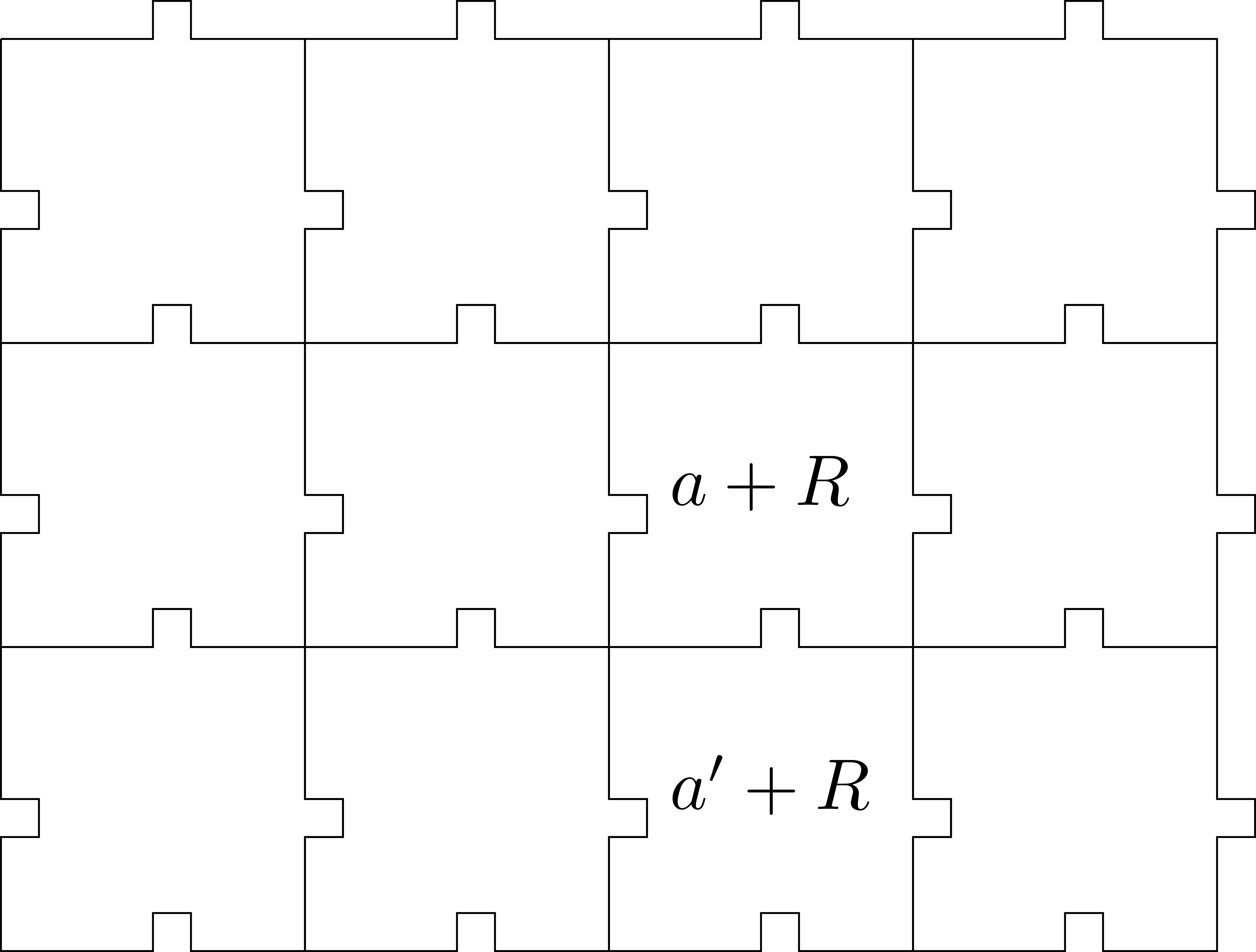}}
    \caption{A tiling by the rigid tile $R$ constructed in Lemma \ref{rigid-tile}.}
    \label{fig:rigid}
\end{figure}

\begin{proof}  As a first guess, one could take $R$ to be the rectangle
$$ R_0 \coloneqq \{ 0,\dots,N_1-1\} \times \dots \times \{0,\dots,N_k-1\}.$$
For this choice of $R$ we certainly have that that every coset $h+\Lambda$ solves the tiling equation $\Tile(R_0;\Z^k)$:
$$ (h+\Lambda) \oplus R_0 = \Z^k.$$
However, the tiling $\Lambda \oplus R_0 = \Z^k$ is not rigid, and it is possible to ``slide'' portions of this tiling to create additional tilings (cf. Example \ref{simple-tile}). To fix this we need to add\footnote{This basic idea of using bumps to create rigidity goes back to Golomb \cite{golomb}.} and remove some ``bumps'' to the sides of $R_0$ to prevent sliding.  There are many ways to achieve this; we give one such way as follows.  For each $j=1,\dots,k$, let $n_j$ be an integer with $2 \leq n_j \leq N_j-3$ (the bounds here are in order to keep the ``bumps'' and ``holes'' we shall create from touching each other).  We form $R$ from $R_0$ by deleting the elements
$$ (n_1,\dots,n_{j-1},0,n_{j+1},\dots,n_k)$$
from $R_0$ for each $j=1,\dots,k$, and then adding the points
$$ (n_1,\dots,n_{j-1},N_j,n_{j+1},\dots,n_k)$$
back to compensate.  Because $R$ was formed from $R_0$ by shifting some elements of $R_0$ by elements of the lattice $\Lambda$, we see that $\Lambda \oplus R = \Lambda \oplus R_0 = \Z^k$.  By translation invariance, we thus have the inclusion
$$\Z^k / \Lambda \subset \Tile(R;\Z^k)_\Universe.$$
It remains to prove the converse inclusion.  Suppose that $A \in  \Tile(R;\Z^k)_\Universe$, thus $A \subset \Z^k$ and $A \oplus R = \Z^k$.  Then for any $a \in A$ and $1 \leq j \leq k$, the point
$$ a + (n_1,\dots,n_{j-1},0,n_{j+1},\dots,n_k)$$
fails to lie in $a+R$ and thus must lie in some other translate $a'+R$ of $R$ for some $a' \in A$ such that $a'+R$ is disjoint from $a+R$.  From the construction of $R$ it can be shown after some case analysis (and is also visually obvious, see Figure \ref{fig:rigid}) that the only possible choice for $a'$ is $a' = a - N_j e_j$, where $e_1,\dots,e_k$ are the standard basis of $\Z^k$.  Thus the set $A$ is closed under shifts by negative integer linear combinations of $N_1 e_1, \dots, N_k e_k$.  If two elements $a, a'$ of $A$ lie in different cosets of $\Lambda$, then $A$ would contain the set
$$ \{ a, a'\} \oplus \{ -c_1 N_1 e_1 - \dots - c_k N_k e_k: c_1,\dots,c_k \in \N \}$$
which has density strictly greater than $\frac{1}{N_1 \dots N_k} = \frac{1}{|R|}$ in the lower left quadrant.  This contradicts the tiling equation $A \oplus R = \Z^k$.  Thus $A$ must lie in a single coset $y+\Lambda$ of $\Lambda$.  Since we have $(y+\Lambda) \oplus R = \Z^k = A \oplus R$, we must then have $A = y+\Lambda$, giving the desired inclusion.
\end{proof}

Now we can prove the following analogue  of \propref{tile-function}.

\begin{proposition}[Equivalence of tiling equations in $G\times\Z^k$ and functional equations]\label{tile-function'} 
Let $G$ be an explicitly finitely generated abelian group, and $G_1=\Z_{N_1}\times\dots\times \Z_{N_k}$ be an explicit finite abelian group with $N_1,\dots,N_k \geq 5$. Let $J,M \geq 1$ be standard natural numbers, and $\Lambda, R$ be as in \ref{rigid-tile}.
Suppose that for each $j=1,\dots,J$ and $m=1,\dots,M$ one is given a (possibly empty) finite subset $H_j^{(m)}$ of $G$ and a (possibly empty) subset $F_j^{(m)}$ of $\Z^k$.  For each $m=1,\dots,M$, assume also that we are given a subset $E_1^{(m)}$ of $G_1$ and let $E^{(m)} \coloneqq \pi^{-1}(E_1^{(m)})$, where $\pi \colon \Z^k \to G_1$ is the quotient homomorphism (with kernel $\Lambda$).
 We adopt the abbreviations
$$ [[a]] \coloneqq \{a\} \times R ; \quad [[a,b]] \coloneqq  \{ n \in \Z: a\leq n \leq b\} \times R \subset \Z\times \Z^k $$
for integers $a \leq b$.  Let $N>J$. Then the following are equivalent:
\begin{itemize}
\item[(i)]  The system of tiling equations
\begin{equation}\label{tile-1'}
\Tile\left( \left( -H_j^{(m)} \times \{0\} \times F_j^{(m)} \uplus \{0\} \times [[j]])\right)_{j=1}^J; \tilde E^{(m)} \right)
\end{equation}
for all $m=1,\dots,M$, together with the tiling equations
\begin{equation}\label{tile-2'}
\Tile\left( \left( \{0\} \times [[\sigma(j)]] \right)_{j=1}^J; G \times ([[1,J]]\oplus N\Z\times\Lambda ) \right)
\end{equation}
for every permutation $\sigma \colon \{1,\dots,J\} \to \{1,\dots,J\}$ of $\{1,\dots,J\}$, admit a solution, where
$$ \tilde E^{(m)} \coloneqq G \times (N\Z \times E^{(m)} \uplus [[1,J]] \oplus N\Z\times \Lambda ).$$
\item[(ii)]  There exist $f_j \colon G \to G_1$ for $j=1,\dots,J$ that obey the system of functional equations
\begin{equation}\label{func-eq-stack'}
\biguplus_{j=1}^J \biguplus_{h_j \in H_j^{(m)}} (F_j^{(m)} + f_j(n+h_j)) = E_1^{(m)} 
\end{equation}
for all $n \in G$ and $m=1,\dots,M$.
\end{itemize}
\end{proposition}

\begin{proof}
The proof of the direction  (ii) implies (i) is similar to the proof of this direction of \propref{tile-function}, with the only difference that the solution defined there should be pulled back, i.e., one should set
$$A_j \coloneqq \biguplus_{n\in G} \{n\} \times N\Z\times \pi^{-1}(\{f_j(n)\}) \subset G \times N\Z \times \Z^k$$
for $j=1,\dots,J$ to construct the desired solution to the system \eqref{tile-1'}, \eqref{tile-2'}.

We turn to prove (i) implies (ii). Let $A_1,\dots,A_J\subset G\times \Z\times \Z^k$ be a solution to the systems \eqref{tile-1'}, \eqref{tile-2'}.
As in the proof of 
\propref{tile-function}, by adapting the argument from the proof of \thmref{combine'} once again, one can show that $A_j\subset G \times N\Z\times \Z^k$.  For any $n \in G$ and $j=1,\dots,J$, if we then define the slice $A_{j,n} \subset \Z^k$ by the formula
$$ A_{j,n} \coloneqq \{ y \in \Z^k: (n,0,y) \in A_j \}$$
we conclude from \eqref{tile-2'} that
$$ A_{j,n} \oplus R = R \oplus \Lambda $$
which from Lemma \ref{rigid-tile} implies that $A_{j,n}$ is a coset of $\Lambda$, or equivalently that
$$ A_{j,n} = \pi^{-1}(f_j(n))$$
for some $f_j(n) \in G_1$. If one now inspects the $G \times \{0\} \times \Z^k$ slice of \eqref{tile-1}, we see that for any $m=1,\dots,M$ one has
$$ \biguplus_{j=1}^J \biguplus_{h_j \in H_j^{(m)}} A_{j,n+h_j} \oplus F_j^{(m)} = E^{(m)} $$
which gives \eqref{func-eq-stack'} upon applying $\pi$. This completes the derivation of (ii) from (i).
\end{proof}

The proof of \propref{tile-function'} is valid in every universe $\Universe^*$ of ZFC, so in particular the problem in \propref{tile-function'}(i) is undecidable if and only if the one in  \propref{tile-function'}(ii) is.  Hence, to prove Theorem \ref{main-red'}, it will suffice to establish the following analogue of  \thmref{main-function}, in which $\Z^2\times \Z_2$ is pulled back to $\Z^2 \times \Z$.

\begin{theorem}[An undecidable system of functional equations in $\Z^2\times \Z$]\label{main-function'}  There exists an explicit finite abelian group $G_0$, a standard integer $M \geq 1$, and for each $m=1,\dots,M$ there exist (possibly empty) finite subsets $H_1^{(m)}, H_2^{(m)}$ of $\Z^2 \times \Z$ and (possibly empty sets) $F_1^{(m)}, F_2^{(m)}, E^{(m)} \subset G_0$ for $m=1,\dots,M$ such that the question of whether there exist  functions $g_1,g_2 \colon \Z^2 \times \Z \to G_0$ that solve
the system of functional equations
\begin{equation}\label{func-eq'}
\biguplus_{h_1 \in H_1^{(m)}} (F_1^{(m)} + g_1(n+h_1)) \uplus
\biguplus_{h_2 \in H_2^{(m)}} (F_2^{(m)} + g_2(n+h_2)) = E^{(m)}
\end{equation}
for all $n \in \Z^2\times \Z$ and $m=1,\dots,M$ is undecidable (when expressed as a first-order sentence in ZFC).
\end{theorem}

We can now prove this theorem, and hence Theorem \ref{main'}, using Theorem \ref{main-hamming}:

\begin{proof}  We repeat the arguments from Section \ref{hamming-sec}.  Let $N$, $D$, $M$, $h_1^{(m)}$, $h_2^{(m)}$, $F_1^{(m)}$, $F_2^{(m)}$, $E^{(m)}$ be as in  \thmref{main-hamming}.  We recall the systems \eqref{sys-1}, \eqref{sys-2} of functional equations, introduced in \secref{hamming-sec}.

As before, for each $d=1,\dots,D$, let $\pi_d \colon \Z_N^D \to \Z_N$ denote the $d^{\mathrm{th}}$ coordinate projection.
We write elements of $\Z^2 \times \Z_2$ as $(n,t)$ with $n \in \Z^2$ and $t \in \Z_2$ and   elements of $\Z^2 \times \Z$ as $(n,z)$ with $n \in \Z^2$ and $z \in \Z$.

For a pair of functions $ g_1, g_2 \colon \Z^2 \times \Z \to \Z_N^D$, consider the system of functional equations
\begin{equation}\label{sys-1'}
\left(\pi_d^{-1}(\{0\}) +  g_j(n,z)\right) \uplus 
\left(\pi_d^{-1}(\{0\}) +  g_j(n,z+1)\right) = \pi_d^{-1}(\{-1,1\}) 
\end{equation}
for $d=1,\dots,D$ and $j=1,2$, as well as the equations
\begin{equation}\label{sys-2'}
\left(F_1^{(m)} +  g_1((n,z)+(h_1^{(m)},0))\right) \uplus
\left(F_2^{(m)} +  g_2((n,z)+(h_2^{(m)},0))\right) = E^{(m)}
\end{equation}
for $m=1,\dots,M$.  It will suffice to establish (using an argument valid in every universe of ZFC) the equivalence of the following two claims:

\begin{itemize}
\item[(i)]  There exist functions $\tilde f_1, \tilde f_2 \colon \Z^2 \times \Z_2 \to \Z_N^D$ solving the systems \eqref{sys-1}, \eqref{sys-2}.
\item[(ii)] There exist functions $ g_1,  g_2 \colon \Z^2 \times \Z \to \Z_N^D$ solving the systems \eqref{sys-1'}, \eqref{sys-2'}.
 \end{itemize}
 
 Indeed, if (i) is equivalent to (ii), by \secref{hamming-sec}, (ii) is equivalent to the existence of functions $f_1,f_2\colon \Z^2\to \{-1,1\}^D$ solving the system \eqref{func-eq-1}. Hence \thmref{main-hamming} implies \thmref{main-function'}.
 
 It  therefore remains to show that (i) and (ii) are equivalent.
 Suppose first that $\tilde f_1, \tilde f_2 \colon \Z^2 \times \Z_2 \to \Z_N^D$ solve the systems \eqref{sys-1}, \eqref{sys-2}. Then we can define $ g_1,  g_2 \colon \Z^2 \times \Z \to \Z_N^D$
 \begin{equation*}
     g_j(n,z)=\tilde f_j(n, z\mod 2),\quad j=1,2
 \end{equation*}
 which solve systems \eqref{sys-1'}, \eqref{sys-2'}.
 Conversely, if  $ g_1,  g_2 \colon \Z^2 \times \Z \to \Z_N^D$ solve the systems \eqref{sys-1'}, \eqref{sys-2'}, then the functions $\tilde f_1, \tilde f_2 \colon \Z^2 \times \Z_2 \to \Z_N^D$ defined by
$$
\tilde f_j(n,t) = (-1)^t g_j(n,0)
$$
solve the systems \eqref{sys-1}, \eqref{sys-2}. The claim  therefore follows.
\end{proof}

\section{Single tile versus multiple tiles}\label{single-multi}

In this section we continue the comparison between tiling equations for a single tile $J=1$, and for multiple tiles $J>1$.  In the introduction we have already mentioned the ``dilation lemma'' \cite[Proposition 3.1]{BH}, \cite[Lemma 3.1]{GT}, \cite{tij} that is a feature of tilings of a single tile $F$ that has no analogue for tilings of multiple tiles $F_1,\dots,F_J$.  Another distinction can be seen by taking the Fourier transform.  For simplicity let us consider a tiling equation of the form $\Tile(F_1,\dots,F_J; \Z^D)$.  In terms of convolutions, this equation can be written as
$$ \one_{A_1} * \one_{F_1} + \dots + \one_{A_J} * \one_{F_J} = 1.$$
Taking distributional Fourier transforms, one obtains (formally, at least)
$$ \ft{\one_{A_1}} \ft{\one_{F_1}} + \dots + \ft{\one_{A_J}} \ft{\one_{F_J}} = \delta$$
where $\delta$ is the Dirac distribution.  When $J>1$, this equation reveals little about the support properties of the distributions $\ft{\one_{A_j}}$.  But when $J=1$, the above equation becomes
$$ \ft{\one_{A}} \ft{\one_{F}} = \delta$$
which now provides significant structural information about the Fourier transform of $\one_A$; in particular this Fourier transform is supported in the union of $\{0\}$ and the zero set of $\ft{\one_F}$ (which is a trigonometric polynomial).  
Such information is consistent with the known structural theorems about tiling sets arising from a single tile; see e.g., \cite[Remark 1.8]{GT}.  Such a rich structural theory does not seem to be present when $J \geq 2$.

Now we present a further structural property of tilings of one tile that is not present for tilings of two or more tiles, which we call a ``swapping property''.  We will only state and prove this property for one-dimensional tilings, but it is conceivable that analogues of this result exist in higher dimensions.

\begin{theorem}[Swapping property]\label{Swap}  Let $G_0$ be a finite abelian group, and for any integers $a, b$ we write
$$[[a]] \coloneqq \{a\} \times G_0 \subset \Z \times G_0$$
and
$$[[a,b]] \coloneqq \{n \in \Z: a \leq n \leq b \} \times G_0 \subset \Z \times G_0.$$
Let $A^{(0)}, A^{(1)}$ be subsets of $\Z \times G_0$ which agree on the left in the sense that 
 $$A^{(0)} \cap [[n]] =  A^{(1)} \cap [[n]]$$
whenever $n \leq -n_0$ for some $n_0$.  Suppose also that there is a finite subset $F$ of $\Z \times G_0$ such that
\begin{equation}\label{a0-a1}
A^{(0)} \oplus F = A^{(1)} \oplus F.
\end{equation}
Then we also have
$$ A^{(\omega)} \oplus F = A^{(0)}
\oplus F$$
for any function $\omega \colon \Z \to \{0,1\}$, where 
\begin{equation}\label{a-omega}
A^{(\omega)} \coloneqq \bigcup_{n \in \Z} A^{(\omega(n))} \cap [[n]]
\end{equation}
is a subset of $\Z \times G_0$ formed by mixing together the fibers of $A^{(0)}$ and $ A^{(1)}$.
\end{theorem}

\begin{proof}  
For any $n \in \Z$ and $j=0,1$, we define the slices $A^{(j)}_n, F_n \subset G$ by the formulae
$$ A^{(j)}_n \coloneqq \{ x \in G_0: (n,x) \in A^{(j)} \}$$
and
$$ F_n \coloneqq \{ x \in G_0: (n,x) \in F \}.$$
By inspecting the intersection (or ``slice'') of \eqref{a0-a1} at $[[n]]$ for some integer $n$, we see that
$$
\biguplus_{l \in \Z} A^{(0)}_{n-l} \oplus F_l =  \biguplus_{l \in \Z} A^{(1)}_{n-l} \oplus F_l.
$$
(Note that all but finitely many of the terms in these disjoint unions are empty.)
In terms of convolutions on the finite abelian group $G_0$, this becomes
$$
\sum_{l \in \Z} \one_{A^{(0)}_{n-l}} * \one_{F_l}(x) =  \sum_{l \in \Z} \one_{A^{(1)}_{n-l}} * \one_{F_l}(x)
$$
for all $n \in \Z$ and $x \in G_0$.
If one now introduces the functions $f_n \colon G_0 \to \C$ for $n \in \Z$ by the formula
$$ f_n \coloneqq \one_{A^{(1)}_n} - \one_{A^{(0)}_n}$$
then by hypothesis $f_n$ vanishes for $n \leq n_0$, and also
\begin{equation}\label{fan}
\sum_{l \in \Z} f_{n-l} * \one_{F_l}(x) = 0
\end{equation}
for every $n \in \Z$ and $x \in G$.  

To analyze this equation we perform Fourier analysis on the finite abelian group $G_0$.  Let $\hat G_0$ be the Pontryagin dual of $G_0$, that is to say the group of homomorphisms $\xi \colon x \mapsto \xi \cdot x$ from $G_0$ to $\R/\Z$.  For any function $f \colon G_0 \to \C$, we define the Fourier transform $\hat f(\xi) \colon \hat G_0 \to \C$ by the formula
$$ \ft{f}(\xi) \coloneqq \sum_{x \in G_0} f(x) e^{-2\pi i \xi\cdot x}.$$
Applying this Fourier transform to \eqref{fan}, we conclude that
\begin{equation}\label{job}
 \sum_{l \in \Z} \ft{f_{n-l}}(\xi) \ft{\one_{F_l}}(\xi) = 0
 \end{equation}
for all $n \in \Z$ and $\xi \in \hat G_0$.

Suppose $\xi \in \hat G_0$ is such that $\ft{\one_{F_l}}(\xi)$ is non-zero for at least one integer $l$.  Let $l_\xi$ be the smallest integer with $\ft{\one_{F_{l_\xi}}}(\xi) \neq 0$, then we can rearrange \eqref{job} as
$$ \ft{f_n}(\xi) = -  \sum_{l=1}^\infty \frac{\ft{\one_{F_{l_\xi+l}}}(\xi)}{\ft{\one_{F_{l_\xi}}}(\xi)} \hat{f_{n-l}}(\xi)$$
for all integers $n$.  Since $\ft{f_n}(\xi)$ vanishes for all $n \leq n_0$, we conclude from induction that $\ft{f_n}(\xi)$ in fact vanishes for all $n$.

To summarize so far, for any $\xi \in \hat G_0$, either $\ft{\one_{F_l}}(\xi)$ vanishes for all $l$, or else $\ft{f_n}(\xi)$ vanishes for all $n$.  In either case, we see that we can generalize \eqref{job} to
$$ \sum_{l \in \Z} \omega(n-l) \ft{f_{n-l}}(\xi) \ft{\one_{F_l}}(\xi) = 0$$
for all $n \in \Z$ and $\xi \in \hat G_0$.  Inverting the Fourier transform, this is equivalent to
$$ \sum_{l \in \Z} \omega(n-l) f_{n-l} * \one_{F_l}(x) = 0$$
for all $n \in \Z$ and $x \in G_0$, which is in turn equivalent to
$$
\sum_{l \in \Z} \one_{A^{(0)}_{n-l}} * \one_{F_l}(x) =  \sum_{l \in \Z} \one_{A^{(\omega(n-l))}_{n-l}} * \one_{F_l}(x)
$$
and hence
$$ \biguplus_{l \in \Z} A^{(0)}_{n-l} \oplus F_l = 
 \biguplus_{l \in \Z} A^{(\omega(n-l))}_{n-l} \oplus F_l$$
 for all $n \in \Z$.  This gives \eqref{a-omega} as desired.
\end{proof}

\begin{example}  Let $G_0 = \Z_2$, $F = \{0\} \times \Z_2$, and let
$$ A^{(j)} \coloneqq \{ (n, a^{(j)}(n)): n \in \Z\}$$
for $j=0,1$, where $a^{(0)}, a^{(1)} \colon \Z \to \Z_2$ are two functions that agree at negative integers.  Then we have $A^{(0)} \oplus F = A^{(1)} \oplus F = \Z \times G_0$.  Furthermore, for any $\omega \colon \Z \to \{0,1\}$, the set
$$ A^{(\omega)} \coloneqq \{ (n, a^{(\omega)}(n)): n \in \Z\}$$
satisfies the same tiling equation:
$$ A^{(\omega)} \oplus F = A^{(0)} \oplus F = A^{(1)} \oplus F = \Z \times G_0.$$
\end{example}

\begin{example}  Let $G_0 = \Z_2$, $F = \{ (0,0), (1,1)\}$, and let
$$ A^{(j)} \coloneqq \{ (n, j): n \in \Z\}$$
for $j=0,1$.  Then, as in the previous example, we have $A^{(0)} \oplus F = A^{(1)} \oplus F = \Z \times G_0$.  But for any non-constant function $\omega \colon \Z \to \{0,1\}$, the set
$$ A^{(\omega)} \coloneqq \{ (n, a^{(\omega)}(n)): n \in \Z\}$$
will \emph{not} obey the same tiling equation:
$$ A^{(\omega)} \oplus F \neq A^{(0)} \oplus F = A^{(1)} \oplus F = \Z \times G_0.$$
The problem here is that $A^{(0)}, A^{(1)}$ do not agree to the left.  Thus we see that this hypothesis is necessary for the theorem to hold.
\end{example}

Informally, Theorem \ref{Swap} asserts that if $E \subset \Z \times G_0$ for a finite abelian group $G_0$ and $F$ is a finite subset of $\Z \times G_0$, then the solution space $\Tile(F;E)_\Universe$ to the tiling equation $\Tile(F;E)$ has the following ``swapping property'': any two solutions in this space that agree on one side can interchange their fibers arbitrarily and remain in the space.  This is quite a strong property that is not shared by many other types of equations.  Consider for instance the simple equation
\begin{equation}\label{fg}
f_2(n+1) = -f_1(n)
\end{equation}
constraining two Boolean functions $f_1,f_2 \colon \Z \to \{-1,1\}$; this is a specific case of the equation \eqref{func-eq-3}. We observe that this equation does \emph{not} obey the swapping property.  Indeed, consider the two solutions $(f^{(0)}_1, f^{(0)}_2), (f^{(1)}_1, f^{(1)}_2)$ to \eqref{fg} given the formula 
$$ f^{(i)}_j(n) = (-1)^{1_{n > i+j}}$$
for $i=0,1$ and $j=1,2$.  These two solutions agree on the left, but for a given function $\omega \colon \Z \to \{0,1\}$, the swapped functions
$$ f^{(\omega)}_j(n) = (-1)^{1_{n > \omega(n)+j}}$$
only obeys \eqref{fg} when $\omega(1)=\omega(2)$.    Because of this, unless the equations \eqref{fg} are either trivial or do not admit any two different solutions that agree on one side, it does not seem possible to encode individual constraints such as \eqref{fg} inside tiling equations $\Tile(F;E)$ involving a single tile $F$,  at least in one dimension.
As such constraints are an important component of our arguments, it does not seem particularly easy to adapt our methods to construct undecidable or aperiodic tiling equations for a single tile. We remark that in the very special case of \textit{deterministic} tiling equations,  such as the aperiodic tiling equations that encode the construction of Kari in \cite{kari},  this obstruction is not present, for then if two solutions to \eqref{fg}   agree on one side, they must agree everywhere\footnote{For extensive studies of deterministic configurations see \cite{k92,kp,Lukk}. We thank Jarkko Kari for providing us with these references.}.  So it may still be possible to encode such equations inside tiling equations that consist of  one tile.

However, as was shown in the previous sections, we can encode \emph{any} system of equations of the type \eqref{fg} in a system of  tiling equations  involving more than one tile. 

\begin{example}  In the group $\Z \times \Z_4$, the solutions to the system of tiling equations
$$ \Tile( \{(0,0), (0,2)\}, \{(0,1), (0,3)\}; \Z \times \Z_4 );
 \Tile( \{(0,0)\}, \{(-1,0)\}; \Z \times \{-1,1\} ) $$
 can be shown to be given precisely by sets $A_1,A_2 \subset \Z \times \Z_4$ of the form
 $$ A_j = \{ (n, f_j(n)): n \in \Z\}$$
 for $j=1,2$ and functions $f_1,f_2 \colon \Z \to \{-1,1\}$ solving \eqref{fg}.  The above discussion then provides a counterexample that demonstrates that Theorem \ref{Swap} fails when working with a pair of tiles $F_1,F_2$ rather than a single tile.
 \end{example} 
 
The obstruction provided by Theorem \ref{Swap} relies crucially on the abelian nature of $G_0$ (in order to utilize the Fourier transform), suggesting that this obstruction is not present in the nonabelian setting.  This suggestion is validated by the results in Section \ref{nonab-sec} below.

\section{A nonabelian analogue}\label{nonab-sec}

In this section we give an analogue of Theorem \ref{main} in which we are able to use just one tile instead of two, at the cost of making the group $G$ somewhat nonabelian.  The argument will share several features in common with the proof of Theorem \ref{main}, in particular both arguments will rely on Theorem \ref{main-linear} as a source of undecidability (see Figure \ref{fig:logic}).

In order to maintain compatibility with the notation of the rest of the paper we will continue to write nonabelian groups $G$ in additive notation $G = (G,+)$.  Thus, we caution that \textbf{in this section the addition operation $+$ (or $\oplus$) is not necessarily commutative}. 

\begin{example}[Nonabelian additive notation for permutations] Consider the finite group $S_{\Z_4^2} \equiv S_{16}$, the group of permutations $\alpha \colon \Z_4^2 \to \Z_4^2$ on the order $16$ abelian group $\Z_4^2$; this group will play a key role in the constructions of this section.  With our additive notation for groups, we have
$$ \alpha + \beta = \alpha \circ \beta$$
for $\alpha,\beta \in S_{\Z_4^2}$, with $0$ denoting the identity permutation,  $m\alpha$ denoting the composition of $m$ copies of $\alpha$, and $-\alpha$ denoting the inverse of $\alpha$.
\end{example}

The notion of a periodic set continues to make sense for subsets of nonabelian groups (note that every finite index subgroup of $G$ contains a finite index normal subgroup), as does the notation of a tiling equation $\Tile(F;E)$.  Our main result is

\begin{theorem}[Undecidable nonabelian tiling with one tile]\label{onetile}  There exists a group $G$ of the form $G = \Z^2 \times S_{\Z_4^2}^{D} \times G_0$ for some standard natural number $D$ and explicit finite abelian group $G_0$, a finite non-empty subset $F$ of $G$, and a finite non-empty subset $E_0$ of $S_{\Z_4^2}^{D} \times G_0$, such that the nonabelian tiling equation $\Tile(F;\Z^2 \times E_0)$ is undecidable.
\end{theorem}

We will derive this result from Theorem \ref{main-linear} and some additional preparatory results. The main new idea is to encode the Hamming cube $\{-1,1\}^2$ as the solution to a system of tiling equations involving only a single tile in $S_{\Z_4^2}$.  The use of this group $S_{\Z_4^2}$  is ultimately in order to be able to access the reflection permutation $\rho \in \Z_4^2$, which will play a crucial role in encoding the equation \eqref{func-eq-3} using only one tile rather than two.

We first need some additional notation, which we summarize in Figure \ref{fig:group}.  

\begin{definition}[Notation relating to $S_{\Z_4^2}$ and $\Z_4^2$]\ 
\begin{itemize}
\item[(i)]  We let $\rho \in S_{\Z_4^2}$ denote the reflection permutation $\rho(y_1,y_2) \coloneqq (y_2,y_1)$.
\item[(ii)] We define the regular representation $\tau \colon \Z_4^2\to S_{\Z_4^2}$ by the formula
$$ \tau(h)(x) \coloneqq x-h$$
for $h,x \in \Z_4^2$.
\item[(iii)] We define the coordinate function $\pi \colon S_{\Z_4^2} \to \Z_4^2$ by
$$ \pi(\alpha) \coloneqq \alpha^{-1}(0,0),$$
and observe that
\begin{equation}\label{pi-translate}
\pi(\alpha + \beta) = \beta^{-1}(\pi(\alpha))
\end{equation}
for $\alpha,\beta \in S_{\Z_4^2}$.  In particular we have
\begin{equation}\label{pi-regular}
\pi(\alpha + \tau(h)) = \pi(\alpha) + h
\end{equation}
for all $\alpha \in S_{\Z_4^2}$ and $h \in \Z_4^2$.  
\item[(iv)]  We view the Hamming cube $\{-1,1\}^2$ as a coset of the subgroup $(2\Z_4)^2$ in $\Z_4^2$, where $2\Z_4 = \{ 0 \hbox{ mod } 4, 2 \hbox{ mod } 4\}$ is the order two subgroup of $\Z_4$.  We let $B \subset S_{\Z_4^2}$ denote the set 
\begin{equation}\label{B-def}
    B \coloneqq \pi^{-1}(\{-1,1\}^2),
\end{equation} 
and let $K$ be the order two subgroup of $(2\Z_4)^2$ defined by
$$ K \coloneqq \{ (0,0), (0,2) \}.$$
\item[(v)]  A \emph{cycle} in the permutation group $S_{\Z_4^2}$ is a permutation $\sigma \colon \Z_4^2 \to \Z_4^2$ such that there is an enumeration $\alpha_1,\dots,\alpha_{16}$ of $\Z_4^2$ such that $\sigma(\alpha_i) = \alpha_{i+1}$ for all $i=1,\dots,16$ (with the periodic convention $\alpha_{17}=\alpha_1$).  Note that any such cycle generates a cyclic subgroup $\{0, \sigma, 2\sigma, \dots, 15 \sigma\}$ of $S_{\Z_4^2}$ of order $16$.
\item[(vi)] We let $\mathrm{Stab}(\{-1,1\}^2) \equiv S_{12}$ denote the stabilizer group of $\{-1,1\}^2$, that is to say the subgroup of $S_{\Z_4^2}$ consisting of those permutations that act trivially on the Hamming cube $\{-1,1\}^2$.
\end{itemize}
\end{definition}

\begin{figure}
    \centering
    \begin{tikzcd}
&&& B \arrow[r,"\pi",dashed,twoheadrightarrow] \arrow[d,hookrightarrow,dashed] & \{-1,1\}^2 \arrow[loop, in=120,out=60,looseness=2, "\rho"', dashed] \arrow[d,hookrightarrow,dashed] \\
K \arrow[r,hookrightarrow] & (2\Z_4)^2 \arrow[loop, in=120,out=60,looseness=2, "\rho"'] \arrow[r, hookrightarrow] & \Z_4^2 \arrow[loop, looseness=2, in=120,out=60, "\rho"'] \arrow[r, hookrightarrow, "\tau"] & S_{\Z_4^2} \arrow[r, "\pi",dashed,twoheadrightarrow] & \Z_4^2 \arrow[loop,in=240,out=300,looseness=3, "\rho"] \\
& & & \mathrm{Stab}(\{-1,1\}^2) \arrow[u, hookrightarrow]
\end{tikzcd}
    \caption{Maps between various subgroups (or subsets) of $S_{\Z_4^2}$ and $\Z_4^2$.  Solid arrows denote group homomorphisms; hooked arrows denote injections; double-headed arrows denote surjections; and unlabeled hooked arrows denote inclusions.}
    \label{fig:group}
\end{figure}
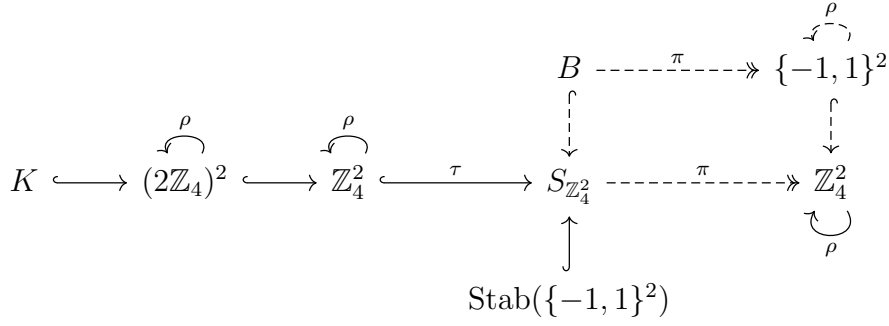

We can now state our preliminary encoding lemma.

\begin{lemma}[Encoding $\{-1,1\}^2$ as a system of tiling equations in $S_{\Z_4^2}$]\label{tiling-system}  Let $A$ be a subset of $S_{\Z_4^2}$.  Then the following are equivalent:
\begin{itemize}
\item[(i)] $A$ is of the form
\begin{equation}\label{aok}
A = \pi^{-1}(\{y\}) = \{ \alpha \in S_{\Z_4^2}: \pi(\alpha) = y \}
\end{equation}
for some $y \in \{-1,1\}^2$.
\item[(ii)]  $A$ obeys the tiling equation
\begin{equation}\label{a-tau}
\Tile( \tau((2\Z_4)^2) ; B )
\end{equation}
 as well as the tiling equations
\begin{equation}\label{a-cycle}
\Tile( \{ \phi, \sigma, 2\sigma, \dots, 15 \sigma\} ; S_{\Z_4^2} )
\end{equation}
for every cycle $\sigma \in S_{\Z_4^2}$ and every $\phi \in \mathrm{Stab}(\{-1,1\}^2)$.
\end{itemize}
\end{lemma}

\begin{proof}  Suppose that (i) holds, thus $A = \pi^{-1}(\{y\})$ for some $Y \in \{-1,1\}^2$.  From \eqref{pi-regular} we then have
$$ A + \tau(h) = \pi^{-1}(\{y+h\})
$$
for every $h \in (2\Z_4)^2$, and hence
$$ A \oplus \tau((2\Z_4)^2) = B;$$
that is to say, \eqref{a-tau} holds.  Similarly, from \eqref{pi-translate} we have
$$ A + \phi = \pi^{-1}(\{\phi^{-1}(y)\}) = \pi^{-1}(\{y\})$$
for every $\phi \in \mathrm{Stab}(\{-1,1\}^2)$, and
$$ A + k \sigma = \pi^{-1}(\{\sigma^{-k}(y)\})$$
for any $k=1,\dots,15$ and every cycle $\sigma \in S_{\Z_4^2}$; since the orbit $y, \sigma^{-1}(y),\dots,\sigma^{-15}(y)$ traverses $\Z_4^2$, we conclude that
$$ A \oplus  \{ \phi, \sigma, 2\sigma, \dots, 15 \sigma\} = S_{\Z_4^2},$$
giving \eqref{a-cycle}.  Thus (i) implies (ii).  

Now suppose conversely that (ii) holds.  Then from \eqref{a-tau} we have $A \subset B$, and moreover for each $\beta \in B$ there exists a unique element of the coset $\beta + \tau((2\Z_4)^2)$  that lies in $A$.

If $\phi$ is an arbitrary element of $\mathrm{Stab}(\{-1,1\}^2)$ and $\sigma \in S_{\Z_4^2}$ is an arbitrary cycle, we see from two applications of \eqref{a-cycle} that
$$ A \oplus \{ \phi, \sigma,2\sigma,\dots,15 \sigma\} = A \oplus \{ 0, \sigma,2\sigma,\dots,15 \sigma\}$$
which on cancelling the terms involving $\sigma$ gives
$$ A \oplus \{\phi\} = A \oplus \{0\}.$$
That is to say, the set $A$ is invariant with respect to the right action of the group $\mathrm{Stab}(\{-1,1\}^2)$.

If $\alpha \in A$, then $\alpha \in B$, and hence $\alpha(\{-1,1\}^2)$ must contain the origin $(0,0)$.  Let $\alpha, \alpha' \in A$ be such that the images $\alpha(\{-1,1\}^2)$, $\alpha'(\{-1,1\}^2)$ intersect only at the origin.  We claim that this implies that $\pi(\alpha)=\pi(\alpha')$.  Indeed, suppose for contradiction that $\pi(\alpha) \neq \pi(\alpha')$.  Then the map $\sigma_0 \colon \alpha(\{-1,1\}^2) \to \alpha'(\{-1,1\}^2)$ defined by
$$ \sigma_0(\alpha(y)) \coloneqq \alpha'(y)$$
contains no fixed points (the only possible fixed point would be at the origin, but the assumption $\pi(\alpha) \neq \pi(\alpha')$ prohibits this).  Since the domain $\alpha(\{-1,1\}^2)$ and range $\alpha'(\{-1,1\}^2)$ of this map only intersect at one point, $\sigma_0$ also contains no cycles, and thus one can complete $\sigma_0$ to a cycle $\tilde \sigma \colon \Z_4^2 \to \Z_4^2$ of $\Z_4^2$.  By construction, the permutations $\tilde \sigma + \alpha$ and $\alpha'$ agree on $\{-1,1\}^2$, thus
$$ \tilde \sigma + \alpha = \alpha' + \phi$$
for some $\phi \in \mathrm{Stab}(\{-1,1\}^2)$.  Defining 
$$ \sigma \coloneqq (-\alpha) + \tilde \sigma + \alpha$$
to be the conjugate of $\tilde \sigma$ by $\alpha$, we see that $\sigma$ is a cycle with
$$ \alpha + \sigma = \alpha' + \phi,$$
but this contradicts the tiling equation \eqref{a-cycle}.  Thus $\pi(\alpha)=\pi(\alpha')$ as claimed.

Now suppose let $\alpha,\alpha'$ be arbitrary elements of $A$, dropping the requirement that $\alpha(\{-1,1\}^2)$, $\alpha'(\{-1,1\}^2)$ intersect only at the origin.  The cardinality of $\alpha(\{-1,1\}^2) \cup \alpha'(\{-1,1\}^2)$ is at most seven; since $\Z_4^2$ has order $16$, we can then certainly find a four-element subset $X$ of $\Z_4^2$ that intersects $\alpha(\{-1,1\}^2) \cup \alpha'(\{-1,1\}^2)$ only at the origin.  We can then find $\beta \in B$ such that $\beta(\{-1,1\}^2)$ only intersects $\alpha(\{-1,1\}^2) \cup \alpha'(\{-1,1\}^2)$ at the origin.  Since the coset $\beta + \tau((2\Z_4)^2)$ intersects $A$, we conclude that there exists $\alpha'' \in A$ in this coset such that $\alpha''(\{-1,1\}^2) = \beta(\{-1,1\}^2)$ only intersects $\alpha(\{-1,1\}^2) \cup \alpha'(\{-1,1\}^2)$ at the origin.  By the previous discussion, we have $\pi(\alpha) = \pi(\alpha'')$ and $\pi(\alpha') = \pi(\alpha'')$, hence $\pi(\alpha) = \pi(\alpha')$.  We conclude that $\pi$ is constant on $A$, thus there exists $y \in \{-1,1\}^2$ such that 
$$ A \subset \{ \alpha \in S_{\Z_4^2}: \pi(\alpha) = y \}.$$
Observe that the right-hand side has cardinality $15!$, while from \eqref{a-cycle} $A$ must have cardinality exactly $16!/16 = 15!$.  Thus we must have equality here, giving (i) as claimed.
\end{proof}

We lift this lemma from $S_{\Z_4^2}$ to the slightly larger group $S_{\Z_4^2} \times \Z_4^2$, to make the encoding of $\{-1,1\}^2$ more visible:

\begin{corollary}[Encoding $\{-1,1\}^2$ as a system of tiling equations in $S_{\Z_4^2} \times \Z_4^2$]\label{tiling-system-2}  Let $A$ be a subset of $S_{\Z_4^2} \times \Z_4^2$.  Then the following are equivalent:
\begin{itemize}
\item[(i)] $A$ is of the form
\begin{equation}\label{aok-2}
A = \pi^{-1}(\{y\}) \times \{y\} = \{ (\alpha,y): \alpha \in S_{\Z_4^2}, \pi(\alpha) = y \} 
\end{equation}
for some $y \in \{-1,1\}^2$.
\item[(ii)]  $A$ obeys the tiling equation
\begin{equation}\label{a-tau-2}
\Tile( \{ (\tau(h), h): h \in (2\Z_4)^2 \} ; \{ (\alpha,\pi(\alpha)): \alpha \in B \} )
\end{equation}
as well as the tiling equations
\begin{equation}\label{a-cycle-2}
\Tile( \{ \phi, \sigma, 2\sigma, \dots, 15 \sigma\} \times \Z_4^2  ; S_{\Z_4^2} \times \Z_4^2 )
\end{equation}
for every cycle $\sigma \in S_{\Z_4^2}$ and every $\phi \in \mathrm{Stab}(\{-1,1\}^2)$.
\end{itemize}
\end{corollary}

\begin{proof}  If (i) holds, then from \eqref{pi-regular} we have
$$ A + (\tau(h),h) = \pi^{-1}(\{y+h\}) \times \{y+h\}$$
for all $h \in (2\Z_4)^2$, which gives \eqref{a-tau-2}, while from \eqref{pi-translate} one has
$$ A \otimes \{\phi\} \times \Z_4^2 = \pi^{-1}(\{y\}) \times \Z_4^2$$
for all $\phi \in \mathrm{Stab}(\{-1,1\}^2)$ and
$$ A \otimes \{k\sigma\} \times \Z_4^2 = \pi^{-1}(\{\sigma^{-k}(y)\}) \times \Z_4^2$$
for any cycle $\sigma \in S_{\Z_4^2}$ and $k=1,\dots,15$, which gives \eqref{a-cycle-2} much as in the proof of the previous lemma.  Thus (i) implies (ii).

Now suppose conversely that (ii) holds.  From \eqref{a-tau-2}  we see that $A$ is contained in the set on the right-hand side of \eqref{a-tau-2}; in particular $A$ is a graph
$$ A = \{ (\alpha,\pi(\alpha)): \alpha \in A' \}$$
for some $A' \subset S_{\Z_4^2}$.  Since $A$ satisfies the tiling equations \eqref{a-tau-2}, \eqref{a-cycle-2}, $A'$ satisfies the tiling equations 
\begin{equation*}
\{ (\alpha + \tau(h), \pi(\alpha) + h): \alpha \in A', h \in (2\Z_4)^2\} = \{ (\alpha,\pi(\alpha)): \alpha \in B \}
\end{equation*}
and 
\begin{equation*}
(A'  \oplus \{ \phi, \sigma, 2\sigma, \dots, 15 \sigma\} ) \times  \Z_4^2  = S_{\Z_4^2} \times \Z_4^2.
\end{equation*}
We conclude  that $A'$ must obey the tiling equations \eqref{a-tau}, \eqref{a-cycle}.  Applying Lemma \ref{tiling-system}, we see that $A'$ is of the form \eqref{aok} for some $y \in \{-1,1\}^2$, and we obtain (i) as required.
\end{proof}

We enumerate the system \eqref{a-tau-2}, \eqref{a-cycle-2} as the system of tiling equations
\begin{equation}\label{tfe}
\Tile( F_\ell; E_\ell ); \quad \ell=1,\dots, L
\end{equation}
for some explicit collection $F_1,\dots,F_L, E_1,\dots,E_L$ of subsets of $S_{\Z_4^2} \times \Z_4^2$ (indeed one has $L = 1 + 15! \times 12!$).  Thus the sets \eqref{aok-2} are precisely the solutions to the tiling system \eqref{tfe}:
\begin{equation}\label{tfe-encode}
 \bigcap_{\ell=1}^L \Tile( F_\ell; E_\ell )_\Universe = \{ \pi^{-1}(\{y\}) \times \{y\}: y \in \{-1,1\}^2 \}.
 \end{equation}
Thus we have successfully encoded the Hamming cube $\{-1,1\}^2$ as a system of tiling equations, in a manner that allows the reflection map $\rho \in S_{\Z_4^2}$ to interact with this encoding.

We now use the above corollary to encode the solvability question appearing in Theorem \ref{main-linear}.

\begin{proposition}[Encoding linear equations]\label{linear-encoding}  Let $D \geq D_0 \geq 1$ and $M_1, M_2 \geq 1$ be natural numbers, and let  $a_{j,d}^{(m)} \in \Z$ be integer coefficients for $j=1,2$, $d=1,\dots,D$, $m=1,\dots,M_j$, and shifts $h_d \in \Z^2$ for $d=1,\dots,D_0$.  Let $N$ be a multiple of $4$ that is sufficiently large depending on all previous data.  We  define the coordinate projections
\begin{align*}
\pi'_1,\dots,\pi'_D &\colon (\Z_N^2)^D \to \Z_N^2 \\
\pi''_1,\dots,\pi''_{D} &\colon (S_{\Z_4^2})^{D} \to S_{\Z_4^2} \\
\pi'''_1, \pi'''_2 &\colon \Z_N^2 \to \Z_N
\end{align*}
in the obvious fashion, while also letting $\Pi \colon \Z_N^2 \to \Z_4^2$ be the reduction mod $4$ map, which is a homomorphism with kernel $(4\Z_N)^2$; see Figure \ref{fig:maps}.  
Then the following statements are equivalent:
\begin{itemize}
\item[(i)]  There exist functions $f_{j,d} \colon \Z^2 \to \{-1,1\} \subset \Z$ for $j=1,2$ and $d=1,\dots,D$ that solve
the system of linear functional equations \eqref{func-eq-2} for all $n \in \Z^2$, $j=1,2$, and $m = 1,\dots,M_j$, as well as the system of linear functional equations \eqref{func-eq-3} for all $n \in \Z^2$ and $d=1,\dots,D_0$.
\item[(ii)]  There exists a set $A \subset \Z^2 \times \Z_2 \times (\Z_N^2)^D \times (S_{\Z_4^2})^D$ that simultaneously solves the following systems of nonabelian tiling equations:

\begin{itemize}
    \item[1. ] The tiling equations \begin{equation}\label{bigtile-1}
 \Tile( \{((0,0),0)\} \times H_j^{(m)} \times C_\sigma  ; \Z^2 \times \Z_2 \times H_j^{(m)} \times S_{\Z_4^2}^{D} )
\end{equation}
for all $j=1,2$ and $m=1,\dots,M_j$, and cycles $\sigma \in S_{\Z_4^2}$, where $H_j^{(m)} \leq (\Z_N^2)^D$ is the subgroup
\begin{equation}\label{hjm}
H_{j}^{(m)} \coloneqq \left\{ (y_{1,d},y_{2,d})_{d=1}^D \in (\Z_N^2)^D: \sum_{d=1}^D a_{j,d}^{(m)} y_{j,d} = 0 \right\}
\end{equation}
and $C_\sigma \subset S_{\Z_4^2}^{D}$ is the set
$$ C_\sigma \coloneqq (\pi''_1)^{-1}(\{0, \sigma, \dots, 15 \sigma \}) = \{0, \sigma, \dots, 15 \sigma \} \times S_{\Z_4^2}^{D-1}.$$
\item[2. ] The  tiling equations
\begin{equation}\label{bigtile-2}
\begin{split}
&\Tile( \{(0,0)\} \times \Z_2 \times (\pi'''_j \circ \pi'_d)^{-1}(\{0\}) \times C_\sigma  ; \\
&\quad\quad \Z^2 \times \Z_2 \times (\pi'''_j \circ \pi'_d)^{-1}(\{-1,1\}) \times S_{\Z_4^2}^{D} )
\end{split}
\end{equation}
for all $d=1,\dots,D$, $j=1,2$, and cycles $\sigma \in S_{\Z_4^2}$. 
\item[3. ] The  tiling equations
\begin{equation}\label{bigtile-3}
\Tile( (T_d \uplus T'_d)  ; \Z^2 \times \Z_2 \times (\Z_N^2)^D \times (\pi''_d)^{-1}(B) )
\end{equation}
for all $d=1,\dots,D_0$, where 
\begin{align*}
T_d &\coloneqq
\{ ((0,0), 0)\} \times (\Z_N^2)^D \times (\pi''_d)^{-1}(\tau(K))\\
T'_d &\coloneqq \{(-h_d,0)\} \times (\Z_N^2)^D \times (\pi''_d)^{-1}(\rho + \tau(K)).
\end{align*}
\item[4. ] The  tiling equations
\begin{equation}\label{bigtile-4}
\Tile( \{((0,0),0)\} \times F_{\ell,d}  ; \Z^2 \times \Z_2 \times  E_{\ell,d} )
\end{equation}
for all $d=1,\dots,D$ and $\ell=1,\dots,L$, where
\begin{align*}
F_{\ell,d} &\coloneqq \{ (y,\zeta) \in (\Z_N^2)^D \times S_{\Z_4^2}^{D}: (\pi''_d(\zeta), \Pi(\pi'_d(y))) \in F_\ell \} \\
E_{\ell,d} &\coloneqq \{ (y,\zeta) \in (\Z_N^2)^D \times S_{\Z_4^2}^{D}: (\pi''_d(\zeta), \Pi(\pi'_d(y))) \in E_\ell \}
\end{align*}
and $F_\ell, E_\ell$ are the sets from \eqref{tfe}.
\end{itemize}
\end{itemize}
\end{proposition}

\begin{figure}
    \centering
    \begin{tikzcd}
    & H_j^{(m)} \arrow[d,dashed,hookrightarrow] & & C_\sigma \arrow[d,dashed,hookrightarrow] \\
    \{-1,1\}^{2D} \arrow[r,dashed,hookrightarrow] \arrow[d,"\pi'_d",dashed,twoheadrightarrow] & (\Z_N^2)^D \arrow[d,"\pi'_d",twoheadrightarrow] & & (S_{\Z_4^2})^D \arrow[d,"\pi''_d",twoheadrightarrow] \\
    \{-1,1\}^2 \arrow[d,dashed,twoheadrightarrow,"\pi'''_j"] \arrow[r,hookrightarrow,dashed] & \Z_N^2 \arrow[d,twoheadrightarrow,"\pi'''_j"] \arrow[r,"\Pi",twoheadrightarrow] & \Z_4^2 & S_{\Z_4^2} \arrow[l, "\pi"', twoheadrightarrow,dashed]\\
    \{-1,1\} \arrow[r,hookrightarrow,dashed] & \Z_N & \tau(K) \arrow[ur, hookrightarrow, dashed] & \rho + \tau(K) \arrow[u, hookrightarrow,dashed] & B \arrow[ul, hookrightarrow, dashed]
\end{tikzcd}
    \caption{Some of the sets and maps mentioned in Proposition \ref{linear-encoding}. (The notation is the same as in \figref{fig:group}.)}
    \label{fig:maps}
\end{figure}
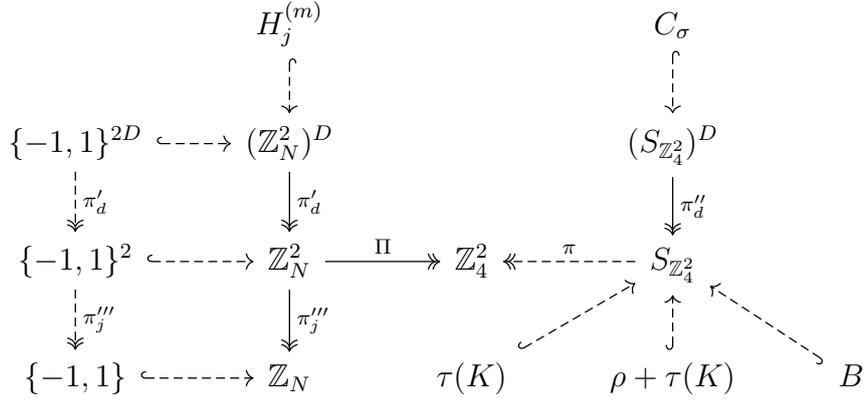

\begin{proof} Suppose that (i) holds.  The sets 
$$ 
\pi^{-1}(\{y\}) = \{ \alpha \in S_{\Z_4^2}: \pi(\alpha) = y \} $$
 have the same cardinality $15!$ for all $y \in \{-1,1\}^2$, so we may arbitrarily enumerate
$$ 
 \pi^{-1}(\{y\}) =\{ \alpha_{y,1},\dots,\alpha_{y,15!} \}$$
for each $y \in \{-1,1\}^2$ and some distinct permutations $\alpha_{y,k}$ for $y \in \{-1,1\}^2$, $k=1,\dots,15!$.  We then let $A$ denote the set of all elements of $\Z^2 \times \Z_2 \times (\{-1,1\}^2)^D\times (S_{\Z_4^2})^D$ of the form
$$ (n, t, (y_{n,t,d})_{d=1}^D, (\alpha_{y_{n,t,d},k})_{d=1}^{D} )$$
for $(n,t) \in \Z^2 \times \Z_2$ and $k=1,\dots,15!$, where
\begin{equation}\label{def-y-ntd}
    y_{n,t,d} \coloneqq ((-1)^t f_{1,d}(n), (-1)^t f_{2,d}(n)) \in \{-1,1\}^2.
\end{equation} 

We now verify the tiling equations \eqref{bigtile-1}, \eqref{bigtile-2}, \eqref{bigtile-3}, \eqref{bigtile-4}.  For any $(n,t) \in \Z^2 \times \Z_2$ and any cycle $\sigma \in S_{\Z_4^2}$, we see from Lemma \ref{tiling-system} that
$$ \{ (\alpha_{y_{n,t,d},k})_{d=1}^{D}: k=0,\dots,15!\} \oplus C_\sigma = S_{\Z_4^2}^{D}$$
and thus for any $d=1,\dots,D$, $j=1,2$, $\sigma$ one has
$$ A \oplus \{((0,0),0)\} \times H_j^{(m)} \times C_\sigma
= \biguplus_{(n,t) \in \Z^2 \times \Z_2}
\{(n,t)\} \times ((y_{n,t,d})_{d=1}^D + H_j^{(m)}) \times S_{\Z_4^2}^{D}.$$
From \eqref{hjm}, \eqref{def-y-ntd}, \eqref{func-eq-2} we have
$$ (y_{n,t,d})_{d=1}^D + H_j^{(m)} = H_j^{(m)}$$
and the equation \eqref{bigtile-1} then follows.  In a similar vein, the set
$$ A \oplus \{(0,0)\} \times \Z_2 \times (\pi'''_j \circ \pi'_d)^{-1}(\{0\}) \times C_\sigma $$
for a given $d=1,\dots,D$, $j=1,2$, $\sigma$ is equal to
$$ \biguplus_{n \in \Z^2}
\{n\} \times \Z_2 \times 
\left( \biguplus_{t \in \Z_2} (\pi'''_j \circ \pi'_d)^{-1}(\pi'''_j(y_{n,t,d}))\right) \times S_{\Z_4^2}^{D}.$$
From \eqref{def-y-ntd} we have
$$\{ \pi'''_j(y_{n,0,d})\} \uplus \{ \pi'''_j(y_{n,1,d}) \} = \{-1,1\}$$
and the equation \eqref{bigtile-2} then follows.  Turning now to \eqref{bigtile-3}, we see from the definitions of $A, T_d, T'_d$ that the set
$$ A \oplus (T_d \uplus T'_d)$$
is equal to
$$ \biguplus_{(n,t) \in \Z^2 \times \Z_2} \{(n,t)\} \times (\Z_N^2)^D \times
(\pi''_d)^{-1}\left( A_{n,t,d} \oplus \tau(K) \uplus A_{n+h_d,t,d} \oplus (\rho + \tau(K)) \right)$$
where $A_{n,t,d} \subset S_{\Z_4^2}$ is the set
$$ A_{n,t,d} \coloneqq \{ \alpha_{y_{n,t,d},k}: k=1,\dots,15!\} = \pi^{-1}(\{ y_{n,t,d}\}).$$
From \eqref{pi-regular} we have that
$$A_{n,t,d} \oplus \tau(K) = \pi^{-1}(y_{n,t,d} + K)$$
and similarly from \eqref{pi-translate}, \eqref{pi-regular}  (and the involutive nature of $\rho$) that
$$A_{n+h_d,t,d} \oplus (\rho + \tau(K)) = \pi^{-1}( \rho(y_{n+h_d,t,d}) + K ).$$
On the other hand, from the equation \eqref{func-eq-3} we have
$$ (y_{n,t,d} + K) \uplus (\rho(y_{n+h_d,t,d}) + K) = \{-1,1\}^2$$
and hence
$$ A \oplus (T_d \uplus T'_d) = \biguplus_{(n,t) \in \Z^2 \times \Z_2} \{(n,t)\} \times (\Z_N^2)^D \times
(\pi''_d)^{-1}\left( \pi^{-1}(\{-1,1\}^2) \right).$$
Since $\pi^{-1}(\{-1,1\}^2) = B$, this gives \eqref{bigtile-3}.

Finally we verify \eqref{bigtile-4}.  Suppose that
$(n,t,y,\zeta) \in \Z^2 \times \Z_2 \times (\Z_N^2)^D \times S_{\Z_4^2}^{D}$ is an element of
$$ A \oplus \{((0,0),0)\} \times F_{\ell,d}.$$
By the definition of $A$ and $F_{\ell,d}$, we thus have
$$ (n,t,y,\zeta) = (n,t,(y_{n,t,d})_{d=1}^D + y', (\alpha_{y_{n,t,d},k})_{d=1}^D + \zeta')$$
for some $k=1,\dots,15!$, $y' \in (\Z^2_N)^D$, and $\zeta' \in S_{\Z^4}^D$ obeying $(\pi''_d(\zeta'), \Pi(\pi'_d(y'))) \in F_\ell$.  In particular, we have
$$ \Pi(\pi'_d(y)) = \Pi(y_{n,t,d}) + \Pi(\pi'_d(y'))$$
and
$$ \pi''_d(\zeta) = \alpha_{y_{n,t,d},k} + \pi''_d(\zeta')$$
and hence by definition of $A_{n,t,d}, F_\ell$ and
\eqref{tfe-encode} (or Corollary \ref{tiling-system-2})
\begin{equation}\label{pidd}
(\pi''_d(\zeta), \Pi(\pi'_d(y)))
 \in A_{n,t,d} \times \{\Pi(y_{n,t,d})\} \oplus F_\ell = E_\ell
 \end{equation}
(note from \eqref{tfe-encode} that all the sums in the right-hand side of \eqref{pidd} are distinct).  Conversely, if $(n,t,y,\zeta)$ obeys the constraint \eqref{pidd}, we can reverse the above arguments and represent $(n,t,y,\zeta)$ uniquely as an element of $A \oplus \{((0,0),0)\} \times F_{\ell,d}$. We conclude that
$$  \left\{(n,t,y,\zeta) \in \Z^2 \times \Z_2 \times (\Z_N^2)^D \times S_{\Z_4^2}^{D}: (\pi''_d(\zeta), \Pi(\pi'_d(y)))
 \in E_\ell \right\},
$$ 
and \eqref{bigtile-4} follows.  This concludes the derivation of (ii) from (i).

Now suppose conversely that (ii) holds.  For any $(n,t) \in \Z^2 \times \Z_2$, let $A_{n,t} \subset (\Z_N^2)^D \times S_{\Z_4^2}^{D}$ be the fiber
$$ A_{n,t} \coloneqq \{ (y,\zeta) \in (\Z_N^2)^D \times S_{\Z_4^2}^{D}: (n,t,y,\zeta) \in A \}.$$
From the tiling equation \eqref{bigtile-4} we have for every $d=1,\dots,D$ and $\ell=1,\dots,L$ that
$$ A \oplus \{ ((0,0),0)\} \times F_{\ell,d} = \Z^2 \times \Z_2 \times E_{\ell,d}$$
and hence (on restricting to $\{(n,t)\} \times (\Z^2_N)^D \times (S_{\Z^2_4})^D$) we have
$$ A_{n,t} \oplus F_{\ell,d} = E_{\ell,d}$$
for every $(n,t) \in \Z^2 \times \Z_2$.
By the definition of $F_{\ell,d}, E_{\ell,d}$, this implies that the map $(y,\zeta) \mapsto (\pi''_d(\zeta), \Pi(\pi'_d(y)))$ is injective on $A_{n,t}$, and that the image
$$ A'_{n,t,d} \coloneqq \{ (\pi''_d(\zeta), \Pi(\pi'_d(y))): (y,z) \in A_{n,t} \} \subset S_{\Z_4^2} \times \Z_4^2$$
obeys the tiling equations 
$$ A'_{n,t,d} \oplus F_{\ell} = E_{\ell}$$
for all $\ell=1,\dots,L$.  Applying \eqref{tfe-encode} (or Corollary \ref{tiling-system-2}), we conclude that there exists $y_{n,t,d} \in \{-1,1\}^2$ such that
\begin{equation}\label{aptnd}
 A'_{n,t,d} = \{ \alpha \in S_{\Z_4^2}: \pi(\alpha) = y_{n,t,d}\} \times \{y_{n,t,d} \}.
 \end{equation}
 In particular $A'_{n,t,d}$ has cardinality $15!$, hence $A_{n,t}$ has cardinality $15!$ as well.  From \eqref{aptnd} and the definition of $A'_{n,t,d}$, we see that for any $(y,\zeta) \in A_{n,t}$, we have
\begin{equation}\label{pio}
 \Pi(\pi'_d(y)) =  \pi(\pi''_d(\zeta)) = y_{n,t,d}
 \end{equation}
for all $d=1,\dots,D$.

Next, from \eqref{bigtile-2} we have in particular that
$$
A \subset \Z^2 \times \Z_2 \times (\pi'''_j \circ \pi'_d)^{-1}(\{-1,1\}) \times S_{\Z_4^2}^{D}$$
and hence
$$ \pi'''_j \circ \pi'_d( y ) \in \{-1,1\}$$
whenever $(n,t) \in \Z^2 \times \Z_2$, $(y,\zeta) \in A_{n,t}$, $j=1,2$, and $d=1,\dots,D$.  In particular, $\pi'_d(y) \in \{-1,1\}^2$, which when combined with
\eqref{pio} gives
$$ \pi'_d(y) = y_{n,t,d}$$
(where by abuse of notation we view $\{-1,1\}^2$ as embedded in both $\Z_4^2$ and $\Z_N^2$).  Thus we have
\begin{equation}\label{y-eq}
y = (y_{n,t,d})_{d=1}^D
\end{equation}
whenever $(y,\zeta) \in A_{n,t}$.

From \eqref{bigtile-1} we have
$$
A \subset \Z^2 \times \Z_2 \times H_j^{(m)} \times S_{\Z_4^2}^{D}$$
which when combined with \eqref{y-eq} implies that
$$ (y_{n,t,d})_{d=1}^D \in H_j^{(m)}$$
for $j=1,2$ and $m=1,\dots,M_j$, and $(n,t) \in \Z^2 \times \Z_2$.  If we now introduce the boolean functions $f_{j,d} \colon \Z^2 \to \{-1,1\}$ by the formula
\begin{equation}\label{f1d}
(f_{1,d}(n), f_{2,d}(n)) \coloneqq y_{n,0,d}
\end{equation}
for $n \in \Z^2$ and $d=1,\dots,D$, we conclude that
$$ (f_{1,d}(n), f_{2,d}(n))_{d=1}^D \in H_j^{(m)}$$
or equivalently that
$$ \sum_{d=1}^D a_{j,d}^{(m)} f_{j,d}(n) = 0 \hbox{ mod } N$$
for all $n \in \Z^2$, $j=1,2$, and $m=1,\dots,M_j$.  For $N$ large enough, we may drop the reduction modulo $N$ as the left-hand side is bounded independently of $N$, thus
$$ \sum_{d=1}^D a_{j,d}^{(m)} f_{j,d}(n) = 0 $$
in the integers.  This gives \eqref{func-eq-2}.

Next, from \eqref{bigtile-3}  we have
\begin{align*}
&A_{n,t} \oplus (\Z_N^2)^D \times (\pi''_d)^{-1}(\tau(K))\\
\uplus & A_{n+h_d,t} \oplus (\Z_N^2)^D \times (\pi''_d)^{-1}(\rho + \tau(K))
= (\Z_N^2)^D \times (\pi''_d)^{-1}(B)
\end{align*}
for any $(n,t) \in \Z^2 \times \Z_2$ and $d=1,\dots,D_0$.  Applying the projection $\pi''_d$ followed by \eqref{aptnd}, we conclude that
\begin{align*}
&\{ \alpha \in S_{\Z_4^2}: \pi(\alpha) = y_{n,t,d}\} \oplus \tau(K) \\
\uplus&
\{ \alpha \in S_{\Z_4^2}: \pi(\alpha) = y_{n+h_d,t,d}\} \oplus (\rho + \tau(K)) = B.
\end{align*}
Applying \eqref{pi-translate}, \eqref{pi-regular}, \eqref{B-def}, this is equivalent to
$$ (y_{n,t,d} + K) \uplus (\rho(y_{n+h_d,t,d}) + K) = \{-1,1\}^2.$$
Specializing to $t=0$ and using \eqref{f1d}, we obtain
$$ \{f_{1,d}(n)\} \uplus \{ f_{2,d}(n+h_d)\} = \{-1,1\}$$
which is \eqref{func-eq-3}.  This establishes (i).
\end{proof}

By Theorem \ref{main-linear}, there exist choices of $D, D_0, M_1, M_2, \alpha_{j,d}^{(m)}, h_d$ such that the problem in Proposition \ref{linear-encoding}(i) is undecidable in ZFC.  As the proof of this proposition is valid in every universe $\Universe^*$ of ZFC, we conclude that for $N$ a sufficiently large (standard) multiple of $4$, the problem in Proposition \ref{linear-encoding}(ii) is undecidable in ZFC.  Thus, we can find an undecidable system of nonabelian tiling equations
$$ \Tile( \tilde F_\ell; \Z^2 \times \tilde E_\ell ); \quad \ell = 1,\dots,\tilde L$$
for some non-empty subsets $\tilde F_1,\dots,\tilde F_{\tilde L}$ of $\Z^2 \times \Z_2 \times (\Z_N^2)^D \times (S_{\Z_4^2})^D$ and subsets $\tilde E_1, \dots, \tilde E_{\tilde L}$ of $\Z_2 \times (\Z_N^2)^D \times (S_{\Z_4^2})^D$.  Applying Theorem \ref{combine} (and Remark \ref{nonab}), we obtain Theorem \ref{onetile} as desired.

\section{Open problems and remarks}\label{sec-problems}

\subsection{}
Recall that \conjref{ptc} is open in dimensions $d>2$ (see \subsecref{subsub-ptc}  for further discussion and known results). The following question then naturally arises.

\begin{problem}
Let $G$ be a non-trivial finitely generated abelian group. Are there  any finite set $F\subset \Z^2\times G$ and periodic set $E\subset \Z^2\times G$ such that  the tiling equation $\Tile(F,E)$ is aperiodic? 
\end{problem}
We hope to address this problem in a future work.

\subsection{}
\conjref{ptc} was originally formulated in \cite{LW} for $G=\R^d$.   It is an interesting question to determine the precise relationship  between the $\Z^d$ and $\R^d$
formulations of the conjecture.

\begin{problem}\label{problem_Rd}
Let $d\geq 1$. What can be said about  \conjref{ptc} for $G=\R^d$, given that the conjecture holds in $\Z^d$?
\end{problem}
In the one dimensional case, the two formulations are equivalent (see \cite{LW}). In the two dimensional case the precise relationship between the discrete and continuous formulations of the periodic tiling conjecture is not known. In \cite{ken,err}  Kenyon extended the result in \cite{gbn} and proved that the periodic tiling conjecture holds for topological discs in $\R^2$. 
In \cite{GT} we  proved that for any finite $F\subset \Z^2$ and periodic $E\subset \Z^2$, all the solutions to the equation $\Tile(F,E)$ are weakly periodic. This  implies a similar result for some special types of tile $F$ in  $\R^2$, by using the construction in \remref{extend_Rd}.  We hope to extend this class of tiles and consider the higher dimensional case of \probref{problem_Rd} in a future work.

\subsection{}
We suggest several possible improvements of our construction.
\begin{itemize}
    \item It might be possible to modify our argument to allow $E_0$ in \thmref{main} to equal $G_0$.  
\begin{problem}
Is there any finite abelian group $G_0$ for which there exist finite non-empty sets $F_1,F_2\subset \Z^2\times G_0$ such that the tiling equation $\Tile(F_1,F_2;\Z^2\times G_0)$ is undecidable?
\end{problem}

\item In \cite{goodman} a construction of two tiles $F_1,F_2$ in $\R^2$ is given in which the tiling equation is aperiodic if one is allowed to apply arbitrary isometries (not just translations) to the tiles $F_1,F_2$; each tile ends up lying in eight translation classes, so in our notation this is actually an aperiodic construction with $J = 2 \times 8 = 16$.  Similarly for the ``Ammann $A2$'' construction in \cite{ags} (with $J = 2 \times 4 = 8$). The aperiodic tiling of $\R^2$ (or the hexagonal lattice) construction in \cite{socolar-taylor} involves a class of twelve tiles that are all isometric to a single tile (twelve being the order of the symmetry group of the hexagon).

It may be possible to adapt the construction used to prove Theorem \ref{main} so that the tiles $F_1,F_2$ are isometric to each other.  On the other hand, we note a remarkable result of Gruslys, Leader, and Tan \cite{glt} that asserts that for \emph{any} non-empty finite subset $F$ of $\Z^d$, there exists a tiling of $\Z^n$ for some $n \geq d$ by isometric copies of $F$.

\begin{problem}
Does our construction provide an example of a finite abelian group  $G_0$, a subset $E_0\subset G_0$, and two finite sets $F_1,F_2 \subset \Z^2\times G_0$ which are isometric to each other, such that the tiling equation 
$$\Tile(F_1,F_2;\Z^2\times E_0)$$ 
is undecidable?
\end{problem}

\item The finite abelian group $G_0$ in \thmref{main} obtained from our construction is quite large.  
It would be interesting to optimize the size of $G_0$.
\begin{problem}
Find the smallest finite abelian group $G_0$ for which there exist finite non-empty sets $F_1,F_2\subset \Z^2\times G_0$, and  $E_0\subset G_0$ such that the tiling equation $\Tile(F_1,F_2;\Z^2\times E_0)$ is undecidable.
\end{problem}

\item It might be possible  to reduce the dimension $d$ in \thmref{main'} by  ``folding'' more efficiently the finite construction of $G_0$ in \thmref{main}, into a \emph{lower} dimensional infinite space.
\begin{problem}
Let $G_0=\prod_{i=1}^{d} \Z_{N_i}$. Suppose that there exist $F_1,F_2\subset \Z^2\times G_0$ and $E_0\subset G_0$  such that the tiling equation $\Tile(F_1,F_2;\Z^2\times E_0)$ is  undecidable. Does this imply the existence of $d'<2+d$ such that there are finite sets $F'_1, F'_2\subset \Z^{d'}$ and a periodic set $E\subset \Z^{d'}$ for which the tiling equation $\Tile(F'_1,F'_2;E)$ is undecidable?  
\end{problem}

\item In   \remref{algo} we discuss the algorithmic undecidable tiling problem which our argument establishes. In this tiling problem, the finite abelian group $G_0$ is one of the inputs.  It might be that a slight  modification of our construction would imply the existence algorithmic undecidable tiling problem with two tiles in $\Z^2\times G_0$, for a \textit{fixed} finite abelian group $G_0$.
\begin{problem}
Is there any finite abelian group $G_0$ such that the  decision problem of whether the tiling equation $\Tile(F_1,F_2;\Z^2\times E_0)$ is solvable for any given finite subsets $F_1,F_2\subset \Z^2\times G_0$ and $E_0\subset G_0$, is algorithmically undecidable?
\end{problem}
\end{itemize}

\appendix

\section{Undecidability and aperiodicity}\label{wang-app}

In this section we give a well-known argument of Wang  (see \cite{Ber, R}) that undecidability implies aperiodicity (which in particular implies that undecidable tiling equations admit tilings in the standard universe).  The argument is usually phrased in the language of algorithmic undecidability, but can be adapted without difficulty to the logical notion of undecidability discussed here.

\begin{theorem}[Undecidability implies aperiodicity]
Let $G$ be an explicit finitely generated abelian group, $J, M \geq 1$ be standard natural numbers, and for each $m=1,\dots,M$, let $F^{(m)}_1,\dots,F^{(m)}_J$ be finite subsets of $G$, and let $E^{(m)}$ be a periodic subset of $G$.  If the system $\Tile(F^{(m)}_1,\dots,F^{(m)}_J; E^{(m)})$ for $m=1,\dots,M$ is undecidable, then it is aperiodic.
\end{theorem}

\begin{proof}  We will establish the contrapositive: if the system $\Tile(F^{(m)}_1,\dots,F^{(m)}_J; E^{(m)})$ for $m=1,\dots,M$ fails to be aperiodic, then it must be decidable.  By definition of aperiodicity, one of the following two statements must hold:
\begin{itemize}
\item[(i)]  The standard solution set $\bigcap_{m=1}^M \Tile(F^{(m)}_1,\dots,F^{(m)}_J; E^{(m)})_\Universe$ is empty.
\item[(ii)]  The standard solution set $\bigcap_{m=1}^M \Tile(F^{(m)}_1,\dots,F^{(m)}_J; E^{(m)})_\Universe$ contains a periodic tuple $(A_1,\dots,A_J)$.
\end{itemize}
In case (ii), since periodic sets are definable, we have a solution to the system $\Tile(F^{(m)}_1,\dots,F^{(m)}_J; E^{(m)})$, $m=1,\dots,M$ in every universe $\Universe^*$ of ZFC, and hence by the G\"odel completeness theorem the solvability question is decidable (in the positive).  Now suppose that we are in case (i).  By the compactness theorem in logic\footnote{One can also proceed here using K\"onig's lemma, or via other compactness theorems such as Tychonoff's theorem.}, there must therefore exist a finite subset $S$ of $G$ such that the system $\Tile(F^{(m)}_1,\dots,F^{(m)}_J; E^{(m)})$, $m=1,\dots,M$ is not satisfiable in $S$, in the sense that there does not exist $A_1,\dots,A_J \subset G$ such that
$$ ((A_1 \oplus F^{(m)}_1) \cap S) \cup \dots \cup ((A_J \oplus F^{(m)}_J) \cap S) = E^{(m)} \cap S$$
for all $m=1,\dots,M$.  This latter assertion can be viewed as unsatisfiable boolean sentence involving the finite number of propositions $(n \in A_j)$ for $j=1,\dots,J$, $m=1,\dots,M$, and $n \in S - F^{(m)}$.  The unsatisfiability of this sentence can be proven in ZFC (simply by exhausting a truth table), and it implies the unsolvability of $\Tile(F^{(m)}_1,\dots,F^{(m)}_J; E^{(m)})$, $m=1,\dots,M$ in every universe of ZFC.  By the G\"odel completeness theorem, we thus see that the solvability of this system is decidable (in the negative).  The claim follows.
\end{proof}


\end{document}